\documentclass[english,12pt]{amsart}
\usepackage{comment}
\usepackage{float}
\usepackage{mathtools}
\usepackage{amstext}
\usepackage{amsthm}
\usepackage{amsmath}
\usepackage{amssymb}
\usepackage{algorithm}
\usepackage{geometry}
\usepackage[backref=page,colorlinks,citecolor=blue,bookmarks=true]{hyperref}\hypersetup{linkcolor=blue,  citecolor=blue, urlcolor=blue}

\usepackage[nobysame,abbrev,alphabetic]{amsrefs}

\usepackage{tikz}
\usetikzlibrary{shapes}

\usepackage{nicefrac,xfrac}

\floatname{algorithm}{Algorithm}

\makeatletter

\numberwithin{equation}{section}
\numberwithin{figure}{section}
\theoremstyle{plain}
\newtheorem{thm}{\protect\theoremname}[section]
\DeclareMathOperator*{\Ex}{\mathbb{E}}
\DeclareMathOperator*{\Pro}{\mathbb{P}}

\newtheorem{cor}[thm]{Corollary}%{\protect\Corollaryname}
\newtheorem*{thm*}{\protect\theoremname}
\theoremstyle{definition}
\newtheorem{problem}[thm]{\protect\problemname}
\newtheorem*{problem*}{Problem}
\theoremstyle{remark}
\newtheorem*{rem*}{\protect\remarkname}
\theoremstyle{remark}
\newtheorem{rem}[thm]{\protect\remarkname}

\theoremstyle{definition}
\newtheorem{defn}[thm]{\protect\definitionname}
\theoremstyle{plain}
\newtheorem{prop}[thm]{\protect\propositionname}
\theoremstyle{plain}
\newtheorem{fact}[thm]{\protect\factname}
\theoremstyle{definition}

\theoremstyle{plain}
\newtheorem{lem}[thm]{\protect\lemmaname}
\theoremstyle{plain}
\newtheorem{claim}[thm]{Claim}%{\protect\claimname}
\theoremstyle{plain}
\usepackage[backref=page,colorlinks,citecolor=blue,bookmarks=true]{hyperref}\hypersetup{linkcolor=blue,  citecolor=blue, urlcolor=blue}
\usepackage[fontsize=11pt]{scrextend}

\usepackage{algorithmic}

\usepackage{changepage}

\floatname{algorithm}{Algorithm}

\let\originalleft\left
\let\originalright\right
\renewcommand{\left}{\mathopen{}\mathclose\bgroup\originalleft}
\renewcommand{\right}{\aftergroup\egroup\originalright}

\makeatother

\usepackage{babel}

\addto\captionsenglish{\renewcommand{\definitionname}{Definition}}
\addto\captionsenglish{\renewcommand{\factname}{Fact}}
\addto\captionsenglish{\renewcommand{\lemmaname}{Lemma}}
\addto\captionsenglish{\renewcommand{\problemname}{Problem}}
\addto\captionsenglish{\renewcommand{\propositionname}{Proposition}}
\addto\captionsenglish{\renewcommand{\remarkname}{Remark}}
\addto\captionsenglish{\renewcommand{\theoremname}{Theorem}}
\addto\captionsfrench{}
\providecommand{\definitionname}{Definition}
\providecommand{\factname}{Fact}
\providecommand{\lemmaname}{Lemma}
\providecommand{\problemname}{Problem}
\providecommand{\propositionname}{Proposition}
\providecommand{\remarkname}{Remark}
\providecommand{\theoremname}{Theorem}
\providecommand{\examplename}{Example}

\newcommand{\eps}{\varepsilon}

\newcommand{\Sym}{{\rm Sym}}

\newcommand{\defect}{{\rm def}}
\newcommand{\Defect}{{\rm Def}}

\newcommand{\Id}{{\rm Id}}

\newcommand{\cF}{\mathcal{F}}
\newcommand{\FF}{\mathbb{F}}
\newcommand{\cA}{\mathcal{A}}

\newcommand{\cT}{\mathcal{T}}

\newcommand{\cG}{\mathcal{G}}
\newcommand{\cX}{\mathcal{X}}
\newcommand{\cY}{\mathcal{Y}}
\newcommand{\cZ}{\mathcal{Z}}

\newcommand{\Ker}{\textrm{Ker}}

\newcommand{\tr}{\textrm{tr}}
\newcommand{\Hom}{\textrm{Hom}}
\newcommand{\raE}{\overrightarrow{E}}
\newcommand{\raP}{\overrightarrow{P}}

\newcommand{\frError}{{\frak {error}}}
\newcommand{\frF}{{\frak F}}
\newcommand{\frD}{{\frak D}}
\newcommand{\frR}{{\frak R}}
\newcommand{\frP}{ {\frak P}}

\title[Stability of Homomorphisms, Coverings and Cocycles]{Stability of Homomorphisms, Coverings and Cocycles {I}: \\ Equivalence}

\author[M.\ Chapman]{Michael Chapman}
\address{Michael Chapman\hfill\break
	Courant Institute of Mathematical Sciences\hfill\break
	New York University,\hfill\break 251 Mercer St, New York, NY 10012, USA.}
\email{mc9578@nyu.edu}

\author[A.\ Lubotzky]{Alexander Lubotzky}
\address{Alexander Lubotzky\hfill\break
	Weizmann institute of Science\hfill\break
	Rehovot, Israel.}
\email{alex.lubotzky@mail.huji.ac.il}

\begin{document}

\maketitle

\begin{abstract}
    This paper is motivated by recent developments in group stability, high dimensional expansion, local testability of error correcting codes and topological property testing. In Part I, we formulate and motivate three stability problems:
    \begin{itemize}
        \item \emph{Homomorphism stability}: Are almost homomorphisms close to homomorphisms?
        \item \emph{Covering stability}: Are almost coverings of a cell complex close to genuine coverings of it?
         \item \emph{Cocycle stability}: Are $1$-cochains whose coboundary is small close to  $1$-cocycles?
    \end{itemize}
    We then prove that  these three problems are \textbf{equivalent}.
    \\
    
    In Part II of this paper \cite{CL_part2}, we present examples of stable (and unstable) complexes, discuss various applications of our new perspective, and provide a plethora of open problems and further research directions. In another companion paper \cite{CL_stability_Random_complexes}, we study the stability rates of random simplicial complexes in the Linial--Meshulam model.
\end{abstract}

\section{\textbf{Introduction}}
Property testing, a deep and popular area of theoretical computer science, is the study of  algorithmic local to global phenomena. A property tester is a randomized algorithm which aims to infer global properties of an object by applying a few random local checks on it (see Section \ref{sec:Prop_testing}). Property testing plays a major role in some of the most celebrated results in theoretical computer science. These include the PCP theorem \cites{PCP_thm,dinur2007pcp,radhakrishnan2007dinur}, good locally testable codes \cites{LTC_DELLM,dinur2022good,LTC_Panteleev_Kalachev}, and the recent complexity theoretic refutation of Connes' embedding problem \cite{MIPRE}.

Classically, functional properties were studied, such as  linearity \cite{BLR} and being a low-degree polynomial \cite{BFL91}, as well as  properties of graphs such as $k$-colorability \cite{Alon_Kriv_k-colorability}. In recent years, group theorists  began studying whether being a permutation solution to a system of group equations is a testable property (cf. \cites{GlebskyRivera,ArzhantsevaPaunescu,becker2022testability,becker_Mosheiff_Luobtzky2022testability}). For example, given a pair of almost commuting permutations, are they close to a commuting pair? By associating between systems of group equations and group presentations, this problem can be framed as an instance of homomorphism testing. Namely, in the spirit of Ulam's stability problem \cite{Ulam}, this is the study of whether an approximate group homomorphism is close to a genuine one. This framework generalizes the linearity test of \cite{BLR}, as well as the local testability of error correcting codes (cf. \cite{LTC_DELLM}), as we discuss in Section \ref{sec:LTC}. 

Even more recently, Dinur--Meshulam \cite{Dinur-Meshulam} initiated the study of topological property testing. In their paper, they study whether being a topological covering is a testable property. Namely, are near coverings of a given complex close to actual coverings of it? Furthermore, they relate this problem to the first cohomology of the underlying complex, and a new high dimensional expansion parameter of it.

In this paper, we give an alternative setup for testing coverings, and consequently a different cohomological viewpoint from the one used in \cite{Dinur-Meshulam}.\footnote{We describe the difference in Section \ref{sec:Dinur-Meshulam_framework}.} In our setup, the parameter acquired by the cohomological viewpoint is  equivalent to the covering testing parameter (as oppose to \cite{Dinur-Meshulam} in which they only bound one another). Moreover,  both the cohomological and covering parameters are related to the homomorphism testing parameter of the fundamental group of the underlying complex. We intentionally chose to view  these stability problems through the  property testing lens. One can present these problems in a purely algebraic manner, but we believe property testing perspective is important  for future applications, as well as to understand better the mechanisms involved. At the end of this paper, we discuss a slight generalization of our framework, in which our main Theorems \ref{thm:main1},\ref{thm:main2} and \ref{thm:main3} still hold.

The second part of this paper \cite{CL_part2} provides several elementary examples of stable complexes, namely complexes where the local perspective implies a global structure -- almost coverings of it are close to genuine coverings, and thus almost cocycles are close to cocycles and almost homomorphisms of its fundamental group are close to actual homomorphisms of the group. In addition,  large cosystols of a complex are related to property (T) of its fundamental group. Furthermore,  a family of bounded $2$-dimensional coboundary expanders with $\FF_2$ coefficients is provided, resolving a special case of a  problem due to Gromov. Lastly, we discuss several potential avenues of enquiry, including novel ways to tackle (but, not yet solve) the \emph{existance of non-sofic groups} problem (See Problem \ref{prob:sofic_groups}).

    Let us now define our objects of study. As oppose to the body of the paper, here we give a concise presentation of the notions of interest  ---  homomorphism stability, covering stability and cocycle stability ---  in a purely algebraic manner,  with no additional motivation. For a more thorough and motivational presentation of these notions, see Sections \ref{sec:Homomorphism_stability_of_group_presentations}, \ref{sec:Covering_stability_of_polygonal_complexes} and \ref{sec:Cocycle_stability}. In the rest of this section, $\rho\colon \mathbb{R}_{\geq 0}\to \mathbb{R}_{\geq 0}$ is a \emph{rate function} (Definition \ref{defn:Rate_function}), namely it is non-decreasing and satisfies $\rho(\eps)\xrightarrow{\eps\to 0}0$.

\subsection*{Homomrphism stability} Let $\Gamma\cong \langle S|R\rangle$ be a \emph{finite} group presentation, i.e., $S$ and $R$ are finite sets. Let $\Sym(n)$ be the symmetric group acting on $[n]=\{1,...,n\}$. Every map  $f\colon S\to \Sym(n)$ uniquely extends to a homomorphism from the free group $f\colon \cF(S)\to \Sym(n)$. This homomorphism factors through $\Gamma$ if and only if for every $r\in R$ we have $f(r)=\Id$. Hence, our local measurement for being a homomorphism would be how much is this condition violated. Namely, the \emph{homomorphism local defect} of $f$ is 
$$\defect_{hom}(f)=\Pro_{\substack{i\in [n] \\ r\in R}}[f(r).i\neq i]\leq \eps,$$
where the probability is the uniform one.\footnote{Other probability distributions over $R$ can and should be discussed, though in this paper we focus on the uniform one. For more on that, see Section \ref{sec:LTC}, Section \ref{sec:alt_dist} and the second part of this paper \cite{CL_part2}.}  Given  permutations $\sigma \in \Sym(n)$ and $\tau\in \Sym(N)$ where $N\geq n$,  the \emph{normalized Hamming distance with errors} between them   is
\[
d_h(\sigma,\tau)=1-\frac{|\{i\in [n]\mid \sigma(i)=\tau(i)\}|}{N}.
\]
This allows us to measure the distance between maps $f\colon S\to \Sym(n)$ and $g\colon S\to \Sym(N)$ as the expected value of the distances between their evaluation points,\footnote{Again, there are reasons to take other distributions over $S$ that are not the uniform one. See Section \ref{sec:alt_dist} and the second part of this paper \cite{CL_part2} for more on that.} i.e., 
\begin{equation}\label{eq:intro_dist_functions}
    d_h(f,g)=\Ex_{s\in S}[d_h(f(s),g(s))].
\end{equation}
The \emph{homomorphism global defect} of $f$ is its distance to the closest genuine homomorphism, namely
\[
\Defect_{hom}(f)=\inf\{d_h(f,\varphi) \mid N\geq n,\ \varphi\colon S\to \Sym(N),\ \defect_{hom}(\varphi)=0\}.
\]
The presentation $\langle S|R\rangle$ is said to be \emph{$\rho$-homomorphism stable} if $\Defect_{hom}(f)\leq \rho(\defect_{hom}(f))$ for every map $f\colon S\to \Sym(n)$.

\subsection*{Covering stability}
Recall that a \emph{polygonal complex} is a $2$-dimensional cell complex whose $2$-cells are polygons. The most popular example is that of $2$-dimensional simplicial complexes, in which the pasted polygons are all triangles. Another popular example is that of $2$-dimensional cube complexes (cf. \cite{sageev1995ends}), in which all pasted polygons are quadrilaterals. In our setup, polygons are not of any unique or bounded length. 

A combinatorial map $f\colon \cY\to \cX$ between two complexes is a function that maps $i$-cells of $\cY$ to $i$-cells of $\cX$ in an incident preserving manner. Hence, by defining the map on $G(\cY)$, the underlying graph of $\cY$, there is at most one way of extending it to all  paths in $G(\cY)$, and thus to extend it to the polygons of $\cY$. A combinatorial map $f$ from a graph $\cG$ to a complex $\cX$ is a \emph{genuine covering}, if one can add polygons to $\cG$ and extend $f$ accordingly so that $f$ becomes a topological covering. 

Assume $f$ was already a covering of the underlying graph of $\cX$. By the path lifting property of topological coverings (See Proposition 1.30, page 60, in \cite{Hatcher_Alg_Top}), for every closed path $\pi$ in $G(\cX)$ which originates at a vertex $x\in V(\cX)$, and for every vertex $x'$ in the fiber $f^{-1}(x)\subseteq V(\cG)$, one can lift $\pi$ to a path $\pi'$ in $\cG$ which originates at $x'$ and satisfies $f(\pi')=\pi$. Even if the original path $\pi$ was closed, the lifted path $\pi'$ may be either open or closed. Now, $f$ is a genuine covering if and only if for every polygon $\pi$ in $\cX$ and every $x'\in f^{-1}(x)$ the lifted path $\pi'$ is closed. Hence, we can define the \emph{covering local defect} of $f$  to be 
\[
\defect_{cover}(f)=\Pro_{\substack{\pi\in \overrightarrow{P}(\cX) \\ x'\in f^{-1}(x)}}\left[\textrm{The\ lift\ of}\ \pi\  \textrm{to\ } x'\textrm{\ is\ open}\right],
\]
where $\overrightarrow{P}(\cX)$ is the set of oriented polygons in $\cX$. 

There is a natural measure of distance between graphs, which is called the (normalized) \emph{edit distance}. In Section \ref{sec:graphs_a_la_BS} we provide a rigorous definition \eqref{eq:Edit_distance}, but in essence it measures how many edges need to be changed/deleted/added to move from the smaller graph to the larger one. This metric can also be defined on the category of  $\cX$-labeled graphs, namely graphs $\cY$ with a fixed combinatorial map into $\cX$. Thus, the \emph{covering global defect} $\Defect_{cover}(f)$ of $f\colon \cY\to \cX$ would be its normalized edit distance to the closest genuine covering $\varphi\colon \cZ\to \cX$ in the category of $\cX$-labeled graphs. The complex $\cX$ is \emph{$\rho$-covering stable} if $\Defect_{cover}(f)\leq \rho(\defect_{cover}(f))$ for every such $f$.

\subsection*{Cocycle stability}
Let $\cX$ be a polygonal complex. An anti-symmetric map from $\raE(\cX)$, the oriented edges of $\cX$, to $\Sym(n)$ is called a \emph{$1$-cochain with permutation coefficients}. In a similar manner to combinatorial maps, these $1$-cochains extend in a unique way to paths in the underlying graph $G(\cX)$. A $1$-cochain $\alpha$ is said to be a \emph{$1$-cocycle}, if for every (oriented) polygon $\pi$ we have $\alpha(\pi)=\Id$. Hence, we can define the \emph{cocycle local defect} of $\alpha$ to be 
\[
\defect_{cocyc}(\alpha)=\Pro_{\substack{i\in [n] \\ \pi \in \overrightarrow{P}(\cX)}}[\alpha(\pi).i\neq i].
\]
We already defined in \eqref{eq:intro_dist_functions} a measure of distance between maps from a fixed set to (potentially varying) permutation groups. Therefore, we have a measure of distance between $1$-cochains. Then, the \emph{cocycle global defect}  $\Defect_{cocyc}(\alpha)$ is the  distance of $\alpha$ to the closest $1$-cocycle. As before, $\cX$ is \emph{$\rho$-cocycle stable} if $\Defect_{cover}(\alpha)\leq\rho(\defect_{cover}(\alpha))$ for every $1$-cochain $\alpha$.
\\

\subsection*{Our results}
Given a polygonal complex $\cX$ together with a base point $*\in \cX$,  the \emph{fundamental group} $\pi_1(\cX,*)$ is the collection of loops based at $*$ up to homotopy equivalence.  Furthermore, if $*$ is chosen to be a vertex of the underlying graph $G(\cX)$, then by retracting any spanning tree of  $G(\cX)$ we acquire  a presentation of $\pi_1(\cX,*)$ (See Definition \ref{defn:pres_of_fundamental_group}). 
On the other hand, for any group presentation $\Gamma=\langle S|R\rangle$, one can construct its \emph{presentation complex} $\cX_{\langle S|R\rangle}$, whose fundamental group is isomorphic to $\Gamma$ (See Definition \ref{defn:presentation_complex}).
The main goal (of this part) of the paper is to prove the following equivalences.
\begin{thm}  \label{thm:main1}
    Let $\cX$ be a connected polygonal complex. Then, the following are equivalent:
    \begin{itemize}
        \item $\cX$ is $\rho$-cocycle stable.
        \item $\cX$ is $\rho$-covering stable.
    \end{itemize}
    Note that the rate $\rho$ is  \textbf{the same} in both.
\end{thm}
\begin{thm}\label{thm:main2}
    Let $\Gamma \cong \langle S|R\rangle$ be a group presentation with $|S|,|R|<\infty$. Let $\cX_{\langle S|R\rangle}$ be the presentation complex associated with $\langle S|R\rangle$. Then, the following are equivalent:
    \begin{itemize}
        \item $\cX_{\langle S|R\rangle}$ is $\rho$-cocycle stable.
        \item $\langle S|R\rangle$ is $\rho$-homomorphism stable.
    \end{itemize}
    Again, $\rho$ is the exact same rate.
\end{thm}

\begin{thm}\label{thm:main3}
     Let $\cX$ be a connected polygonal complex, and $*$ a vertex of $\cX$. Let $T$ be a spanning tree of $G(\cX)$, and  let $\pi_1(\cX,*)\cong \langle S|R\rangle $ be the  presentation of the fundamental group of $\cX$ associated with retracting $T$ to the basepoint $*$. Then
     \begin{enumerate}
         \item If $\langle S|R\rangle$ is $\rho$-homomorphism stable, then $\cX$ is $\rho$-cocycle stable.
         \item If $\cX$ is $\rho$-cocycle stable, then $\langle S|R\rangle$ is $\left(|\raE(\cX)|\cdot\rho\right)$-homomorphism stable, where $|\raE(\cX)|$ is the number of edges in $\cX$.  
     \end{enumerate}
\end{thm}
Though not as tight as Theorems \ref{thm:main1} and \ref{thm:main2}, Theorem \ref{thm:main3} still implies a qualitative equivalence, and in many contexts  -- see for example the discussion in Section \ref{sec:group_soficity} -- it is enough.

Our notion of homomorphism stability is usually called \emph{flexible pointwise stability in permutations} in the literature (cf. \cite{BeckerLubotzky}), and it is a thoroughly studied topic. Hence, using   clause $(1)$ of Theorem \ref{thm:main3}, there are \emph{off the shelf} examples of complexes with known stability rates. For example, by \cite{levit_lazarovich2019surface}, triangulations of compact surfaces of genus  $g\geq 2$ are $\rho$-stable with rate $\rho(\eps)=\Theta(-\eps \log\eps)$. While, by \cite{BeckerMosheiff}, when $g=1$ the surface is $\rho$-stable with $\rho$ being $\Omega(\sqrt{\eps})$ and $O(\eps^{\frac{1}{4}})$. In the second part of this paper \cite{CL_part2}, we provide elementary examples of complexes with \textbf{linear} stability rate. 

\begin{rem}
    As oppose to Theorem \ref{thm:main1}, Theorems \ref{thm:main2} and \ref{thm:main3} can be framed in much greater generality than presented here. Our proof method applies to any metric coefficient group --- not only permutations equipped with the Hamming metric. We  thus relate flexible Hilbert--Schmidt stability \cites{HadwinShulman,BeckerLubotzky}, Frobenius norm stability \cites{DGLT,LubotzkyOppenheim} and so on to a stability criterion for cocycles. See Corollary \ref{cor:other_Gammas}.
\end{rem}

Part I of the paper is organized as follows: In Section \ref{sec:Prop_testing} we define all property testing related notions. In Section \ref{sec:Homomorphism_stability_of_group_presentations} we define and motivate homomorphism stability of group presentations.  In Section \ref{sec:Covering_stability_of_polygonal_complexes} we define and motivate covering stability of polygonal complexes. In Section \ref{sec:Cocycle_stability} we define and motivate cocycle stability of polygonal complexes. The motivational parts of Sections \ref{sec:Homomorphism_stability_of_group_presentations}, \ref{sec:Covering_stability_of_polygonal_complexes} and \ref{sec:Cocycle_stability} are more involved, and can be skipped by the non-expert or first time reader. Section \ref{sec:equivalences}  is devoted to the proofs of Theorems \ref{thm:main1},\ref{thm:main2} and \ref{thm:main3}.  In Section \ref{sec:alt_dist} we provide a generalization of our framework, in which Theorems \ref{thm:main1},\ref{thm:main2} and \ref{thm:main3} still hold. 

\subsection*{Acknowledgements}

Michael Chapman acknowledges with gratitude the Simons Society of Fellows and is supported by a grant from the Simons Foundation (N. 965535).
Alex Lubotzky is supported by the European Research Council (ERC)
under the European Union's Horizon 2020 (N. 882751), and by a research grant from the Center for New Scientists at the Weizmann Institute of Science.

\section{\textbf{Property testing}}\label{sec:Prop_testing}

The goal of a property tester is to decide whether a given function satisfies a specific property by querying it at a few random evaluation points. Such a tester cannot decide with certainty that the function satisfies the property, since it may require reading all the values of the function. Hence, a property tester is good given that it satisfies the following: As the rejection probability of the test gets smaller, the closer the function in hand is to another function which has the property. We now make this paragraph precise. For a thorough introduction to the theory, see \cite{Goldreich}.

Let $\frD$ and $\frR$ be collections of \textbf{finite} sets. Let $\frF$ be the collection of functions whose domain is from $\frD$ and range is from $\frR$, namely 
\[
\frF=\{f\colon D\to R\mid D\in \frD,R\in \frR\}.
\]
Let $f_1\colon D_1\to R_1$ and $f_2\colon D_2\to R_2$ be two functions from $\frF$. The \emph{Hamming distance with errors} between them is defined as follows:
\begin{equation}\label{eq:Hamming_dist}
    d_H(f_1,f_2)=|D_2\setminus D_1|+|D_1\setminus D_2|+|\{x\in D_1\cap D_2\mid f_1(x)\neq f_2(x)\}|.
\end{equation}
When $D_1=D_2$, this notion agrees with the usual Hamming distance between functions. If for every $f\colon D\to R$ and $x\notin D$ we define $f(x)=\frError$, and we assume that $\frError\notin R'$ for any $R'\in \frR$, then 
we can equivalently define
\[
d_H(f_1,f_2)=|\{x\in D_1\cup D_2 \mid f_1(x)\neq f_2(x)\}|,
\]
which is the usual way the Hamming distance is defined.
Though the Hamming distance with errors is a good metric, it outputs natural numbers and cannot tend to zero without being zero. Hence, we usually normalize it in some way. The most common normalization is the following, which we call the \emph{normalized Hamming distance with errors}:
\begin{equation}\label{eq:normalized_Hamming_distance}
    d_h(f_1,f_2)=\frac{d_H(f_1,f_2)}{|D_1\cup D_2|}=\Pro_{x\in D_1\cup D_2}[f_1(x)\neq f_2(x)],
\end{equation}
where $\Pro_{x\in A}$ is the uniform probability on $A$. As we further discuss in Section \ref{sec:alt_dist}, and specifically in Claim \ref{claim:weighting_system_perturbation}, in certain cases other probability distributions are more natural. But, we stick to the uniform distribution for most of Part I of the paper. 

Let $\frP$ be a subset of $\frF$. We think of  $\frP$ as the collection of functions  in $\frF$ that satisfy a certain property. The $\frP$-\emph{global defect} of a function $f\in \frF$ is 
\[
\Defect_\frP(f)=\inf \{d_h(f,\varphi)\mid \varphi\in \frP\}.
\]

A \emph{tester} for $\frP$ is a randomized algorithm  $\cT$ that receives as input a black box version\footnote{The way $\cT$ interacts with $f$ is by sending to it an evaluation point $x\in D$ and receiving back the value $f(x)\in R.$} of a function $f\colon D\to R$ in $\frF$ and needs to decide whether $f\in \frP$. The way it operates is as follows: 
\begin{itemize}
    \item \textit{Sampling phase}\label{Sampling_Phase}: $\cT$ randomly chooses an evaluation point $x_1\in D$ and sends it to $f$. The function $f$ answers with $f(x_1)$. According to this value, $\cT$ either asks for another evaluation point $x_2\in D$ (which may depend on $f(x_1)$), or moves to the next phase. $\cT$ is allowed to repeat this process for at most $q$ rounds, where $q$ is independent of $f$.\footnote{In certain contexts, and for some applications, $q$ may be allowed to grow with $f$. Mainly, to make the stability rate  $\rho$  (See Definition \ref{defn:completeness_stability_of_tester}) better using parallel repetition (cf. \cite{Raz_Parallel_Rep_paper,Raz_parallel_rep_survey}), or when $f$ is encoding an exponentially larger function in a succinct way (cf. \cite{BFL91}). Actually, this flexibility is crucial for the applications of efficient stability mentioned in Section 6.4 of \cite{CL_part2}. } The maximum number of sampled points $q$ is called \emph{the query complexity of $\cT$}.\footnote{The term \emph{locality} is also commonly used for $q$ (cf. \cite{LTC_DELLM}).}
    \item \textit{Decision phase}: According to the outputs $f(x_1),...,f(x_q)$, the tester $\cT$ decides whether to accept (namely, it decides that $f\in \frP$) or reject (namely, it decides that $f\notin \frP$).
\end{itemize}
The \emph{$\frP$-local defect} of a function $f\in \frF$ is its rejection probability in $\cT$, namely
\begin{equation}\label{eq:P_local_defect}
    \defect_\frP(f)=\Pro[\cT\ \textrm{rejects}\ f], 
\end{equation}
where the probability runs over all possible sampling phases.
\begin{defn} \label{defn:Rate_function}
    A non-decreasing function $\rho\colon \mathbb{R}_{\geq 0}\to  \mathbb{R}_{\geq 0}$ satisfying  $\rho(\eps)\xrightarrow{\eps \to 0}0$ will be called a \emph{rate function}. 
\end{defn}
\begin{defn} \label{defn:completeness_stability_of_tester}
    The tester $\cT$ is \emph{complete} if
\[
\forall f\in \frF\ \colon\ \ \defect_\frP(f)=0 \iff f\in \frP.
\]
Namely, functions that satisfy the property are always accepted, and all other functions are rejected with some positive probability. 

Let $\rho$ be a rate function as in Definition \ref{defn:Rate_function}. The tester $\cT$ is said to be $\rho$-\emph{stable}\footnote{This notion is also called \emph{$\rho$-robustness} and \emph{$\rho$-soundness} in the literature.}, if
\[
\forall f\in \frF\ \colon\ \ \Defect_\frP(f)\leq \rho(\defect_\frP(f)).
\]
\end{defn}

\section{\textbf{Homomorphism stability of group presentations}}\label{sec:Homomorphism_stability_of_group_presentations}
In 1940, Ulam raised the following problem:
\begin{problem}[Ulam's homomorphism stability problem \cite{Ulam}]
    Given an almost homomorphism between two groups, is it close to a genuine homomorphism between them?
\end{problem}
There are many (non-equivalent) ways to formulate Ulam's problem (cf. \cite{Kazhdan,BOT,glebsky2023asymptotic,GowersHatami,BLT,DGLT, BC22}). In this section we study a finite group presentations variant of it.\footnote{
    Our notion of homomorphism stability is usually referred to in the literature as \emph{pointwise flexible stability in permutations}  (see \cite{BeckerLubotzky}).} 

For a positive integer $n$, let $\Sym(n)$ be the symmetric group acting on $[n]=\{1,...,n\}$. Given  permutations $\sigma \in \Sym(n)$ and $\tau\in \Sym(N)$ where $N\geq n$,  the normalized Hamming distance with errors between them (as defined in \eqref{eq:normalized_Hamming_distance}) is
\begin{equation}\label{eq:permutation_normalized_Hamming_distance}
d_h(\sigma,\tau)=1-\frac{|\{i\in [n]\mid \sigma(i)=\tau(i)\}|}{N}=\Pro_{i\in [N]}[\sigma(i)\neq \tau(i)].
\end{equation}

Let $\Gamma \cong\langle S | R\rangle$ be a  finite group  presentation, and $\cF(S)$ the free group with basis $S$. Let  $f\colon S\to \Sym(n)$ be a function. By the universal property of free groups, $f$ has a unique extension to a homomorphism from $\cF(S)$ to $\Sym(n)$, which we denote by $f$ as well. This $f$ factors through $\Gamma$ if and only if for every $r\in R$ we have  $ f(r)=\Id.$
Hence, the collection of functions $f\colon S\to \Sym(n)$ that induce a homomorphism from $\Gamma$ is
\[
\Hom(\Gamma,\Sym)=\{f\colon S\to \Sym(n)\mid n\in \mathbb{N},\ \forall r\in R\colon f(r)=\Id\}.
\]
This suggests the following tester for deciding whether $f\colon S\to \Sym(n)$ induces a homomorphism from $\Gamma$:
\begin{algorithm}[H]
\caption{\quad\texttt{Homomorphism tester}}\label{alg:hom_test}
\textbf{Input:} $f\colon S\to \Sym(n)$. \\
\textbf{Output:} Accept or Reject.

\begin{algorithmic}[1]

\STATE Pick $r\in R$ unifromly at random.
\STATE Pick $i\in [n]$ uniformly at random. 
\STATE{If $f(r).i=i$, then return Accept.}

\STATE {Otherwise, return Reject.}

\end{algorithmic}
\end{algorithm}

\begin{rem}\label{rem:general_distributions_hom_test}
    Instead of sampling $r\in R$ uniformly at random in Algorithm \ref{alg:hom_test}, one can use any fully supported probability measure on $R$. This slight generalization  is important to the discussion in Section \ref{sec:LTC}, and is described in Section \ref{sec:alt_dist}.  Other than these sections, we would not need this generalized version. Furthermore, it is natural in some instances to measure the Hamming distance in a weighted way by choosing a non-unifrom distribution on $S$ (See Section \ref{sec:alt_dist} and \cite{CVY_efficient,evra2016bounded}). But again, we will use only the uniform measure for the Hamming distances in this part of the paper. 
\end{rem}
To fit Algorithm \ref{alg:hom_test} to the setup of Section \ref{sec:Prop_testing}, we need the input function $f$ to be in the form $f\colon S\times [n]\to [n]$ with the extra condition that $f(s,\cdot)\colon [n]\to [n]$ is a permutation for every $s\in S$. Namely, \begin{equation*}
\begin{split}
     &\frF=\{f\colon S\times[n]\to[n]\mid n\in \mathbb{N}, \forall s\in S \colon f(s,\cdot)\in \Sym(n)\},\\
    &\frP=\Hom(\Gamma,\Sym),\\
    &\forall f_1,f_2\in \frF\  \colon \ \ d_h(f_1,f_2)=\Ex_{s\in S}[d_h(f_1(s,\cdot),f_2(s,\cdot))].
\end{split}
\end{equation*}
Hence, the \emph{homomorphism local defect} of $f\colon S\to \Sym(n)$, which is the rejection probability in the above test,  is
\[
\defect_{hom}(f)=\Ex_{r\in R}[d_h(f(r),\Id)].
\]
Moreover, the \emph{homomorphism global defect} of $f\colon S\to \Sym(n)$, which is its distance to the collection $\Hom(\Gamma,\Sym)$, is
\[
\Defect_{hom}(f)=d_h(f,\Hom(\Gamma,\Sym))=\inf\{d_h(f,\varphi)\mid \varphi\in \Hom(\Gamma,\Sym)\}.
\]

\begin{defn}\
Let $\rho$ be a rate function, as in Definition \ref{defn:Rate_function}.  A finite presentation $\Gamma\cong \langle S|R\rangle$ is said to be\emph{ $\rho$-homomorphism stable} if for every  $f\colon S\to \Sym(n)$, we have
	\[
	\Defect_{hom}(f)\leq \rho(\defect_{hom}(f)).
	\]
\end{defn}

The goal of the following Sections, \ref{sec:LTC} and \ref{sec:group_soficity}, is to motivate the notion of $\rho$-homomorphism stability.

\subsection{Locally testable codes}\label{sec:LTC}
In this section, we relate homomorphism stability to the well studied topic of \emph{locally testable codes}. 

The classical example of a homomorphism tester is the Blum--Luby--Rubinfeld (BLR) linearity test \cite{BLR}. In it, the input is a function $f\colon \FF_2^n\to \FF_2$, where $\FF_2=\{0,1\}$ is the field with two elements, and the goal is to decide whether $f$ is linear. To that end, two points $x,y\in \FF_2^n$ are sampled uniformly at random, and $f$ is evaluated  at three  points $x,y$ and $x+y$. The tester accepts if and only if $f(x)+f(y)=f(x+y)$.
This test is $\rho$-stable for the identity rate function $\rho(\eps)=\eps$ (See Section 1.6 in \cite{odonnell2021analysis} for a proof).

The BLR linearity test is a special case of Algorithm \ref{alg:hom_test}. To see that, note that $\FF_2\cong \Sym(2)$. Moreover, $\FF_2^n\cong\langle \{s_x\}_{x\in \FF_2^n}\mid \{s_x\cdot s_y\cdot s_{x+y}^{-1}\}_{x,y\in\FF_2^n}\rangle $, which is the multiplication table presentation of $\FF_2^n$. So, the linearity test \textbf{is} Algorithm \ref{alg:hom_test} under the restriction that the range of $f$ is $\Sym(2)$ and not any $\Sym(N)$.
In \cite{BC22}, it is proved that the multiplication table presentation of \textbf{any} finite group is $\rho$-homomorphism stable with linear rate $\rho(\eps)=3000\eps$.  This generalizes \cite{BLR}, though with worse parameters.

A \emph{linear binary error correcting code} -- or just \emph{code} from now on -- is a linear subspace of a finite vector space over $\FF_2$. E.g., the Hadamard code is the subspace  of linear functions from $\FF_2^n$ to $\FF_2$, out of all functions from $\FF_2^n$ to $\FF_2$. The linearity test is  a complete and $\rho$-stable tester, in the sense of Definition \ref{defn:completeness_stability_of_tester}, for checking whether a given function is a Hadamard code word. Moreover, its query complexity  (which was defined in the \textit{Sampling phase} of the tester in Section \ref{Sampling_Phase}) is $3$.  Codes that have a tester which is complete, $\rho$-stable and with constant query complexity, are called \emph{locally testable}. 

Every matrix $A\in M_{m\times n}(\FF_2)$ defines the code ${\rm Ker}(A)\subseteq \FF_2^n$.\footnote{Such  matrices are usually called  \emph{parity-check matrices}.} Together with a probability distribution $\mu$ over $[m]$, it also defines a tester:

\begin{algorithm}[H]
\caption{\texttt{\quad Matrix tester}}\label{alg:matrix_test}
\textbf{Input:} A column vector $v\in \FF_2^n$. \\
\textbf{Output:} Accept or Reject.

\begin{algorithmic}[1]
\STATE {Sample $i\in [m]$  according to $\mu$. }
\STATE {Let $R_i(A)$ be the $i^{\rm th}$ row of $A$. Then, if  $R_i(A)\cdot v=0$, return Accept.}
\STATE {Otherwise, return Reject.}

\end{algorithmic}
\end{algorithm}
Algorithm \ref{alg:matrix_test} is complete given that $\mu$ is fully supported. If each row of $A$ contains at most $q$ non-zero entries, for some universal constant $q$, then this tester has bounded query complexity.\footnote{Codes with a parity matrix whose rows are of bounded weight are known in the literature as LDPC codes (cf. \cite{Gallager_LDPC,Sipser_Spielman_exp_codes}).} So, the main obstacle is to choose $A$ such that the tester is $\rho$-stable for some desired rate function $\rho$.\footnote{Every code would be $\rho$-stable for \textbf{some} $\rho$. The goal is to find codes with a sufficiently good rate of local testabillity. See, e.g.,  our choice of parameters in the efficient stability section in Part II of this paper \cite{CL_part2}. }
Algorithm \ref{alg:matrix_test} is clearly a generalization of the BLR linearity test. Furthermore, every tester for a linear locally testable code can be assumed to be in this form \cite{ben2003some}. But, the matrix tester is a special case of Algorithm \ref{alg:hom_test}:\footnote{As mentioned in Remark \ref{rem:general_distributions_hom_test}, in Algorithm \ref{alg:hom_test} we chose a relation uniformly at random. To  mimic  Algorithm \ref{alg:matrix_test}, we need to generalize the homomorphism tester so it may sample a relation according to some  distribution other than the uniform one.}
The code $\Ker(A)$ can be viewed as the collection of homomorphism inducing maps from $S=\{x_j\}_{j\in[n]}$ to $\Sym(2)\cong \FF_2$. The relations in this case are  $R=\left\{\prod x_j^{A_{ij}}\right\}_{i\in [m]}$, where the product is always ordered by the index. This demonstrates how general Algorithm \ref{alg:hom_test} is: It encapsulates local testability of error correcting codes as a \emph{special case}. This viewpoint on Algorithm \ref{alg:hom_test} raises fascinating new problems which we further discuss in the open problem section of the second part of this paper \cite{CL_part2} (see also \cite{CVY_efficient}).

\subsection{Group soficity} \label{sec:group_soficity}
In this section, we relate homomorphism stability to the well studied notion of \emph{sofic groups}. As oppose to the general philosophy of this paper, in this section the exact rate $\rho$ is not important. Namely, homomorphism stability is seen as a qualitative and not quantitative property. This viewpoint is of interest when the finitely presented group $\Gamma \cong \langle S|R\rangle$ is \textbf{infinite}.
\begin{fact}[\cite{ArzhantsevaPaunescu}]\label{fact:moving_between_pres_of_same_gp}
Let  $\langle S|R\rangle, \langle S'|R'\rangle$ be two finite presentations of the same group $\Gamma$. Assume  $\langle S|R\rangle$  is $\rho$-homomorphism stable for some rate function $\rho$. Then $\langle S'|R'\rangle$ is  $(C\cdot \rho)$-homomorphism stable for some constant $C\in \mathbb{R}_{>0}$ that depends only on the two presentations.
\end{fact}
Hence, we say that a group $\Gamma$ is \emph{homomorphism stable} if it has \textbf{some} finite presentation $\Gamma \cong \langle S|R\rangle $ which  is $\rho$-homomorphism stable for \textbf{some} rate function $\rho$. Namely, by disregarding the exact rate $\rho$, homomorphism stability may be viewed as a group property and not  a presentation specific property. 

\begin{defn}
    A finitely presented group $\Gamma\cong \langle S|R\rangle$ is said to be \emph{sofic} if there is a sequence of functions $f_n\colon S\to \Sym(n)$ such that 
    \[
        \defect_{hom}(f_n)\xrightarrow{n\to\infty}0, 
    \]
    and for every $w\notin \langle\langle R\rangle \rangle$,\footnote{$\langle\langle R\rangle \rangle$ is the normal subgroup generated by $R$.}
    \[
        \liminf(d_h(f_n(w),\Id))\geq \frac{1}{2}.
    \]
\end{defn}
\begin{problem}[Gromov \cite{Gromov_sofic}, Weiss \cite{Weiss_Sofic}] \label{prob:sofic_groups}
    Are there non-sofic groups?
\end{problem}
The following is an easy to prove, yet quite insightful, observation. 
\begin{prop}[Glebsky-Rivera \cite{GlebskyRivera}]\label{prop:sof+stable=res_fin}
	If $\Gamma$ is homomorphism stable and sofic, then it is residually finite. 
\end{prop}

Soficity is a \emph{group approximation} property. There are other well studied group approximation properties, such as hyperlinearity, property MF and $L^p$-approximations (cf. \cite{Cap_Lup_Sofic_Hyperlinear_book,DGLT}). The only known examples of non-approximable groups were constructed using the (analogous) observation of Proposition \ref{prop:sof+stable=res_fin} (See \cite{DGLT, LubotzkyOppenheim}).
Though more sophisticated  methods for constructing non-sofic groups were recently suggested (cf. \cite{BowenBurton,dogon2023flexible}), 
 they still require the proof of homomorphism stability of certain groups.\footnote{There are also suggested methods for constructing non-sofic groups that \textbf{do not} use homomorphism stability, at least not in such a direct manner. E.g., transforming the answer reduction part of the compression theorem in \cite{MIPRE} into a linear constraint system game (cf. \cite{slofstra_2019}) would imply the existence of a non-sofic group. Problem \ref{prob:sofic_groups} is still a major open problem, with many suggested tackling angles. We chose to emphasize its relation to homomorphism stability.} We further discuss group soficity in Part II of this paper \cite{CL_part2}, and even suggest a new paradigm for potentially resolving Problem \ref{prob:sofic_groups} in the positive.

\section{\textbf{Covering stability of polygonal complexes}}\label{sec:Covering_stability_of_polygonal_complexes}
Let us move now to another stability problem.
\subsection{Graphs a la Bass--Serre}\label{sec:graphs_a_la_BS}
\begin{defn}
    A graph (a la Bass--Serre, see \cite{Trees_Serre,KOLODNER2021595})  $\cX$ consists of the following data:
    \begin{enumerate}
      
    \item A set $V=V(\cX)$ of vertices. 
    \item A set $\overrightarrow{E}=\raE(\cX)$ of directed edges, together with functions $\tau,\iota\colon \raE\to V$ and an involution $\overline{\ \ ^{\ ^{\ }}}\colon \raE\to \raE$ satisfying 
    \[
\forall e\in \raE\  \colon \ \ \tau(\overline{e})=\iota(e).
    \]
    The function $\tau$ is the terminal point (or end point) function, the function $\iota$ is the initial point (or origin point) function and $\overline{\ \ ^{\ ^{\ }}}$ is the reverse edge function. Namely, an edge $e$ can be visualized as $x\xrightarrow{e} y$, where $x=\iota(e)$ and $y=\tau(e)$. Moreover, we have $y\xrightarrow {\bar{e}} x$. 
    \end{enumerate}
\end{defn}
\begin{rem}
Note the following:
\begin{itemize}
    \item A graph (a la Bass-Serre) is essentially an undirected multi-graph. To make this statement precise, define $[e]=\{e,\bar e\}$ and  $E(\cX)=\{[e]\mid e\in \raE(\cX)\}$. Then $(V(\cX),E(\cX))$ is a multi-graph where the endopints of $[e]\in E(\cX)$ are $\iota(e)$ and $\tau(e)$. 

    Furthermore, multi-graphs can be viewed as $1$-dimensional CW complexes (cf. \cite{Hatcher_Alg_Top}) by defining $V(\cX)$ to be the $0$-dimensional cells, $E(\cX)$ to be the $1$-dimensional cells, and with gluing maps induced by $\tau$ and $\iota$. Hence, we can view graphs a la Bass-Serre as topological spaces.
    \item We may abuse notation and write $xy$ for the edge $x\xrightarrow{e} y$, though there may be several edges with the exact same origin and endpoint. 
   
\end{itemize}
\end{rem}
\begin{defn}\label{defn:paths}
    A \emph{path} $\pi$ in a graph $\cX$ is a sequence of edges $e_1...e_\ell$ such that each $e_i\in  \raE$ and for every $1\leq i\leq \ell-1$, we have $\tau(e_i)=\iota(e_{i+1})$. By denoting $x_i:=\tau(e_i)\in V$ and $x_0:=\iota(e_1)\in V$, we can graphically represent the path $\pi$ by 
    \[
x_0\xrightarrow{e_1}x_1\xrightarrow{e_2}...\xrightarrow{e_\ell}x_\ell.
    \]
    The integer $\ell(\pi)=\ell$ is the \emph{length} of the path.
    The path $\pi$ is \emph{non-backtracking} (or reduced) if $e_{i+1}\neq\bar{e}_i$ for every $1\leq i\leq \ell-1$.  For a path $\pi$, its \emph{inverse} or \emph{reverse orientation} is 
    \[
        \bar \pi =x_\ell \xrightarrow{\bar e_\ell}...\xrightarrow{\bar{e}_1}x_0.
    \] The path $\pi$ is \emph{closed} if $x_0=x_\ell$, and \emph{open} otherwise. When $\pi$ is closed, it has $\ell$  shifts $e_k...e_\ell e_1...e_{k-1}$ and $\ell$ inverses $\bar e_{k-1}...\bar e_1\bar e_\ell... \bar e_k$ which are also closed paths in $\cX$. We call all these shifts and inverses the \emph{orientations} of $\pi$, and denote the collection of all of them by $[\pi]$. The path $\pi$ is \emph{cyclically reduced} if all its orientations are non-backtracking.
    
     %   For every $i\in \mathbb{N}$, let $\frE_i$ be the map that outputs the $i^{\rm th}$ edge of a path. Namely, for $\pi=e_1...e_\ell$, we have
 %    \[
%\frE_i(\pi)=\begin{cases}
%e_i & i\leq \ell,\\
%{\frak {error}} & i>\ell.
 %    \end{cases}
  %   \]

\end{defn}

\begin{defn} [Combinatorial maps between graphs]
    A \emph{combinatorial map} $f\colon \cY\to \cX$ between two graphs is a function that maps vertices of $\cY$ to vertices of $\cX$ and edges of $\cY$ to edges of $\cX$, and is compatible with the structural data of the graph. Namely, it preserves initial points, terminal points and edge flips:
        \[
\forall e\in \raE(\cY)\ \colon \ \ \tau_\cX(f(e))=f(\tau_\cY(e)),\ \iota_\cX(f(e))=f(\iota_\cY(e)),\ f(\overline{e})=\overline{f(e)}.
        \] 
\end{defn}

\begin{rem}\label{rem:combi_maps_on_paths}
    Given a combinatorial map $f\colon \cY\to\cX$, we can extend it to paths in $\cY$. Namely, if $\pi=e_1...e_\ell$ is a path in $\cY$, then $f(\pi)=f(e_1)...f(e_\ell)$ is a path in $\cX$.  Moreover, combinatorial maps induce a continuous map between the graphs as topological spaces.
\end{rem}

Recall that a continuous function $f\colon \cY\to \cX$ between two topological spaces is a (topological) \emph{covering} if for every point $x\in \cX$ there is an open neighborhood $U$ of $x$ such that $f^{-1}(U)$ is a disjoint union of open sets, each of which homeomorphic to $U$ via $f$. If $\cX$ is (path) connected, then the \emph{degree} of the covering, which is the cardinality of   $|f^{-1}(x)|$ for any of the points $x\in \cX$,  is well defined (cf. Chapter 1.3 in \cite{Hatcher_Alg_Top}).  When $|f^{-1}(x)|=n<\infty$, we say that $f$ is a degree $n$ covering, or just $n$-covering.

\begin{defn}
    A combinatorial map $f\colon \cY\to \cX$ between graphs is a (combinatorial) \emph{covering} if it is a topological covering.
\end{defn}

\begin{rem}
    Graph coverings\footnote{Graph coverings are sometimes called \emph{graph liftings} in the literature. This notion is a bit confusing because of the path lifting and homotopy lifting lemmas (which are used for example  in the proof sketch of Fact \ref{fact:completeness_of_covering_tester}). Hence, we stick with coverings for these objects.} is an intensely studied object. They were used both as a random model for graphs \cite{Linial_Amit2002random, amit2001random,MSS_Ramanujan,Puder_expander_coverings} and as a building block in a recent construction of good locally testable codes and good quantum codes \cite{LTC_Panteleev_Kalachev}. 
\end{rem}

\begin{fact} \label{fact:char_of_graph_coverings}
    Let $\cX$ and $\cY$ be finite graphs. Let $f\colon \cY\to \cX$ be a combinatorial map between them. Then, $f$ is a covering if and only if for every vertex $y\in V(\cY)$, the star of $y$ is mapped bijectively to the star of $f(y)$. Namely, if $E_y(\cY)=\{e\in \raE(\cY)\mid \tau(e)=y \}$, then $f|_{E_y(\cY)}$ is a bijection onto $E_{f(y)}(\cX)$. 
\end{fact}

 A combinatorial map between graphs $f\colon \cA\to \cX$ is said to be an \emph{embedding} if it is a set theoretic injection. A \emph{subgraph} of $\cX$ is the image of some embedding into it. The \emph{edit distance} between a graph $\cX$ and a subgraph of it $\cA$ is
 \[
 d_{Edit}(\cX,\cA)=|E(\cX)|-|E(\cA)|,
 \]
 namely, this is the number of edges needed to be deleted to move from $\cX$ to $\cA$.
 Given two graphs $\cX$ and $\cZ$, not necessarily embedded in one another, we can define the edit distance between them using the  largest graph $\cA$  which embeds into both $\cX$ and $\cZ$.\footnote{The \emph{largest} in this case is according to the number of edges.} The \emph{edit distance} between $\cX$ and $\cZ$ would be the maximum between their distances to $\cA$. Namely, 
   \begin{equation}\label{eq:Edit_distance}
     d_{Edit}(\cX,\cZ)=\max(d_{Edit}(\cX,\cA),d_{Edit}(\cZ,\cA))=\max(|E(\cX)|,|E(\cZ)|)-|E(\cA)|.
   \end{equation}
   The normalized version of the edit distance will be
   \begin{equation}\label{eq:norm_edit_distance}
        d_{edit}(\cX,\cZ)=1-\frac{|E(\cA)|}{\max(|E(\cX)|,|E(\cZ)|)}.
   \end{equation}

\begin{rem}\label{rem:enc_of_graphs}
    There is an encoding of graphs for which the edit distance \eqref{eq:Edit_distance} agrees with the Hamming distance \eqref{eq:Hamming_dist} (and similarly for the normalized versions): The vertex set $V(\cX)$ will always be  $\{1,...,|V(\cX)|\}$. The edge set $E(\cX)$ will   always be $\{1,...,|E(\cX)|\}$. Then, the function $\iota\times\tau\colon E(\cX)\to V(\cX)\times V(\cX)$, which outputs for each edge its origin and terminus, encodes the graph (note that we essentially chose an orientation for every edge $[e]=\{e,\bar e\}$). There are many encodings of the same graph (one can permute $V(\cX)$ and $E(\cX)$ arbitrarily, as well as choose different orientations of the edges). Hence, the edit distance between two graphs \textbf{is} the Hamming distance between the two most compatible encodings of them. 

    One may argue that we did not choose the most natural encoding of graphs. A multi-graph can be encoded by its adjacency matrix (see \eqref{eq:adjacency_matrix}), and the resulting edit distance will be slightly different --- it would be the sum of $d_{Edit}(\cX,\cA)$ and $d_{Edit}(\cZ,\cA)$ instead of their maximum as in \eqref{eq:Edit_distance}. But, up to a factor of $2$, they are the same.
\end{rem}

Given a graph $\cX$, an \emph{$\cX$-labeled} graph is a graph $\cY$ with a combinatorial map $\Phi_\cY\colon \cY\to \cX$ which we call \emph{the labeling}. A morphism between $\cX$-labeled graphs is a combinatorial map $f\colon \cY\to \cZ$  that commutes with the labelings, namely $\Phi_\cY=\Phi_\cZ \circ f$. Embeddings and subgraphs of $\cX$-labeled graphs can be defined in an analogous way to before, and hence the edit distance is defined on this category.\footnote{Note that by choosing a weighting system $w$ on the edges of $\cX$, we get a weighted edit distance on $\cX$-labeled graphs by paying $w(\Phi(e))$ when changing $e$. This is shortly discussed in Section \ref{sec:alt_dist}.}

\subsection{Polygonal complexes}
As mentioned above, graphs are $1$-dimensional CW complexes. We now wish to add $2$-dimensional cells to the mixture in a combinatorial way that will allow us to study them through a  property testing lens.

\begin{defn}
A \emph{polygonal complex} $\cX$ is a graph with extra data: 
A set $\raP=\raP(\cX)$ of cyclically reduced closed paths in $\cX$. The elements of  $\raP$ are called (oriented) \emph{polygons}. Similar to the way we always include both $e$ and $\bar e$ as edges in $\raE(\cX)$, we assume that all the orientations  of a polygon must be  in $\raP$. Namely, if  $\pi=e_1...e_\ell\in \raP$, then 
\[
\forall 1\leq k\leq \ell\ \colon\ \ e_k...e_\ell e_1...e_{k-1}\in \raP\quad \textrm{and}\quad \bar e_{k-1} ... \bar e_1 \bar e_\ell ...\bar e_k \in \raP.
\]
We denote the set of un-oriented polygons by $P(\cX)=\{[\pi]\mid \pi\in \raP(\cX)\}$.

For visual examples of polygons see Figure \ref{tikz:Polygon}. As oppose to edges, where there may be multiple edges with the same origin and end point, we assume that there are no two polygons with the same pasted path. Namely, a polygon is uniquely defined by its closed path boundary.
  
\begin{figure}
\begin{center}
    \begin{tikzpicture}[scale=0.8]
    \node [draw, color=black, shape=circle] (x) at (6,-3){};
    \node [draw, color=black, shape=rectangle] (Name) at (9,0){$\pi_1$};
    \node [draw, color=black, shape=circle] (y) at (12,-3){};
    \node [draw, color=black, shape=circle] (z) at (6,3){};
    \node [draw, color=black, shape=circle] (w) at (12,3){};
     \node[purple] (T) at (13.2,4){$e_0$};
     \draw[purple, ->, solid] (x)--(y) node[midway,above]{$e_1$};
\draw[purple, ->, solid] (y)--(w) node[midway,right]{$e_2$};
\draw[purple, ->, solid] (w)--(z) node[midway,above]{$e_3$};
\draw[purple, ->, solid] (z)--(x) node[midway,left]{$e_4$};
\draw[purple,->] (w) to  [out=0,in=90,looseness=15] (w) ; 

\node[draw, regular polygon,regular polygon sides=5]{$\pi_2$};
\node [draw, color=black, shape=circle] (A) at (12:3){};
\node [draw, color=black, shape=circle] (B) at (90:3){};
\node [draw, color=black, shape=circle] (C) at (162:3){};
\node [draw, color=black, shape=circle] (D) at (234:3){};
\node [draw, color=black, shape=circle] (E) at (306:3){};
  \draw[cyan, ->, solid] (A)--(B) node[midway,above]{$e'_1$};
  \draw[cyan, ->, solid] (B)--(C) node[midway,above]{$e'_2$};
  \draw[cyan, ->, solid] (C)--(D) node[midway,left]{$e'_3$};
  \draw[cyan, ->, solid] (D)--(E) node[midway,above]{$e'_4$};
  \draw[cyan, ->, solid] (E)--(A) node[midway,right]{$e'_5$};

 \node [draw, color=black, shape=circle] (*) at (4,-5){};
     \draw[purple,->] (*) to  [out=225,in=315,looseness=15] (*) ; \node[purple] at (4,-6.5) {$s_1$};
     \draw[cyan,->] (*) to  [out=330,in=70,looseness=15] (*) ; \node[cyan] at (5.2,-4) {$s_2$};
     \draw[brown,->] (*) to  [out=110,in=200,looseness=15] (*) ; \node[brown] at (2.8,-4) {$s_3$};
    \end{tikzpicture}
    \begin{adjustwidth*}{0em}{0em}
    \caption{Examples of polygons. $\pi_1=e_1e_2e_3e_4$ is a length $4$ polygon. We can also construct length $5$ polygons of the form $\pi=e_1e_2e_0e_3e_4$ or $\pi'=e_1e_2\bar{e_0}e_3e_4$. $\pi_2=e'_1e'_2e'_3e'_4e'_5$ is a length $5$ polygon. Lastly, every word of length $\ell$ in the free group on $\{s_1,s_2,s_3\}$ defines a length $\ell$ polygon on the bouquet of circles.}\label{tikz:Polygon}
    \end{adjustwidth*}
    \end{center}
    \end{figure}
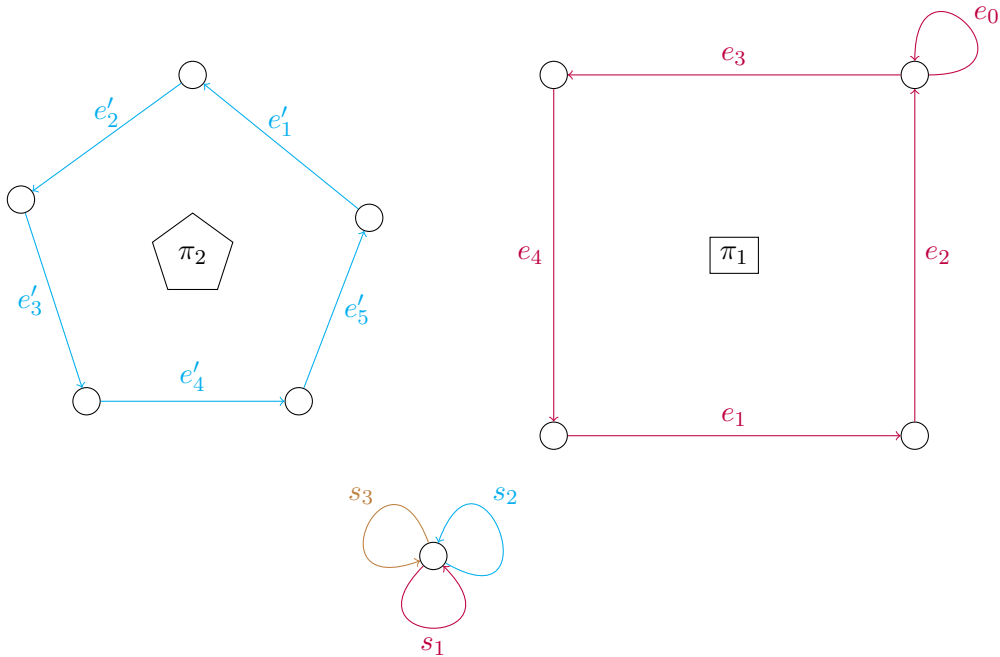
\end{defn}
\begin{rem} Note the following:
\begin{itemize}
    \item  Polygonal complexes are graphs with added data. To be able to distinguish between the graph structure and the whole complex $\cX$, we denote by $G(\cX)$ the \emph{underlying graph} of $\cX$.
    \item Polygonal complexes are $2$-dimensional CW complexes (hence, topological spaces): Their $1$-skeleton is induced by the underlying graph $G(\cX)$. For the $2$-dimensional cells, we paste a disc for every un-oriented polygon $[\pi]$, with the gluing map being according to some representative $\pi$. We therefore use the standard notation for $i$-dimensional cells of a CW complex when suitable:
    \begin{equation}\label{eq:i-cells}
        \cX(0)=V(\cX)\ ,\quad\cX(1)=E(\cX)\ ,\quad \cX(2)=P(\cX),
    \end{equation}
    and 
    \begin{equation}\label{eq:oriented_i-cells}
        \overrightarrow{\cX}(1)=\raE(\cX)\ ,\quad\overrightarrow{\cX}(2)=\raP(\cX)
    \end{equation} for the oriented versions.
\end{itemize}
     
\end{rem}

\begin{defn}
    Let $\cX,\cY$ be two polygonal complexs.  A combinatorial map  $f\colon G(\cY)\to G(\cX)$ is called a \emph{combinatorial map between polygonal complexes}, if for every $\pi \in \raP(\cY)$, we have $f(\pi)\in \raP(\cX)$.
\end{defn}

The notions of embeddings, subcomplexes and coverings generalize to the polygonal complex setup in the natural way.

\subsection{Almost covers}
Let $\cG$ be a graph and $\cX$ a \emph{finite} polygonal complex. Let $f\colon \cG\to G(\cX)$ be a \emph{finite} covering.  Let $\pi=x\xrightarrow{e_1}x_1\xrightarrow{e_2}...\xrightarrow{e_{\ell-1}} x_{\ell -1}\xrightarrow{e_\ell}x$ be a polygon in $\cX$. By the path lifting property of covering maps (See Proposition 1.30, page 60, in \cite{Hatcher_Alg_Top}), for every $x'\in f^{-1}(x)$ there exists a unique path $\pi'$ in $\cG$ that begins at $x'$ and satisfies $f(\pi')=\pi$. The lifted path $\pi'$ may be closed or open.

\begin{fact}\label{fact:completeness_of_covering_tester}
Let $f\colon \cG\to G(\cX)$ be a finite covering. Then, all the lifts of all the polygons in $\raP(\cX)$ are closed if and only if polygons can be added to $\cG$ such that $f$ becomes a covering map of $\cX$  (namely, $\cG$ is the $1$-skeleton of a covering of $\cX$ via $f$). 
\end{fact}
\begin{proof}[Proof idea]
This is essentially the correspondence between (conjugation orbits of) subgroups of the fundamental group $\pi_1(\cX,*)$ and (connected) coverings of $\cX$ (See Theorem 1.38, page 67, in \cite{Hatcher_Alg_Top}). 
\end{proof}

We call a function $f\colon \cG\to G(\cX)$ that can be completed to a covering of $\cX$ (as in Fact \ref{fact:completeness_of_covering_tester}) 
 a \emph{genuine covering} of $\cX$.
This suggests the following tester for whether a graph covering $f\colon \cG\to G(\cX)$ is a genuine covering of $\cX$:\footnote{In Algorithm \ref{alg:Covering_test} we sample a polygon uniformly at random. One could have sampled a polygon using a different distribution. More on that in Section \ref{sec:alt_dist}.}

\begin{algorithm}[H]
\caption{\quad\texttt{Covering tester}}\label{alg:Covering_test}
\textbf{Input:} $f:\cG\rightarrow G(\cX),$ a finite covering of the $1$-skeleton of $\cX$. \\
\textbf{Output:} Accept or Reject.

\begin{algorithmic}[1]

\STATE Pick $[\pi]\in P(\cX)$ uniformly at random. Then, pick a representative $\pi$ of $[\pi]$ uniformly at random. 
\STATE Let $x$ be the starting vertex of $\pi$. Pick $x'\in f^{-1}(x)$ uniformly at random. 
\STATE{If the lift $\pi'$ of $\pi$ to $x'$ is a closed path in $\cG$,
return Accept.}

\STATE {Otherwise, return Reject.}

\end{algorithmic}
\end{algorithm}

Algorithm \ref{alg:Covering_test} may be fitted to the setup of Section \ref{sec:Prop_testing} using an extension of the encoding system suggested in Remark \ref{rem:enc_of_graphs}. In this case: 
\begin{itemize}
    \item $\frF$ is the collection of all graph coverings $f\colon \cG\to G(\cX)$; 
    \item $\frP$ is the collection of all genuine coverings of $\cX$;
    \item Given $f_1\colon \cG_1\to G(\cX)$ and $f_2\colon \cG_2\to G(\cX)$, the normalized Hamming distance \eqref{eq:normalized_Hamming_distance} between them is the normalized edit distance \eqref{eq:norm_edit_distance} between them as $G(\cX)$-labeled graphs. Hence, the \emph{covering global defect} of $f$ is the normalized distance of $\cG$ to the closest genuine covering of $\cX$ in the category of $G(\cX)$-labeled graphs. Namely,
\[
\Defect_{cover}(f)=\inf\left\{d_{edit}(\cG,G(\cY))\middle\vert \varphi\colon \cY\xrightarrow{{\rm covering}}\cX\right\}.
\]
\item As before,  the \emph{covering local defect} of $f$ is the rejection probability in the covering tester (Algorithm \ref{alg:Covering_test}), namely
\[
\defect_{cover}(f)=\Pro_{\substack{[\pi]\in P(\cX)\\ x'\in f^{-1}(x)}}\left[\textrm{The\ lift\ of}\ \pi\  \textrm{to\ } x'\textrm{\ is\ open}\right].
\]
\end{itemize}

\begin{defn}\label{defn:covering_stability}
    Let $\rho$ be a rate function, as in Definition \ref{defn:Rate_function}. A polygonal complex is \emph{$\rho$-covering stable} if for every graph $\cG$ and covering map $f\colon \cG\to G(\cX)$, we have
\[
\Defect_{cover}(f)\leq \rho(\defect_{cover}(f)).
\]
\end{defn}

In the next section, we discuss our motivation to study $\rho$-covering stability, as well as the difference between our covering tester (Algorithm \ref{alg:Covering_test}) and the Dinur--Meshulam one (Algorithm \ref{alg:DM_Covering_test}). 
\subsection{The Dinur--Meshulam framework}\label{sec:Dinur-Meshulam_framework}
In \cite{quantum_soundness_tensor_codes}, Ji--Natarajan--Vidick--Wright--Yuen proved a homomorphism stability result, in unitaries,  that has a very good trade-off between presentation length and  stability rate of the studied groups (see \cite{CVY_efficient}).
In an ongoing project \cite{BCLV_subgroup_tests}, we try to prove a permutation analogue of their result. Covering stability, as in Definition \ref{defn:covering_stability}, arises naturally in our proof strategy.  In \cite{Dinur-Meshulam}, Dinur--Meshulam studied a similar  problem.  The main difference between the Dinur--Meshulam setup and ours is that they think of the degree parameter $n$ as a \textbf{constant}. Namely, in their alternative to Algorithm \ref{alg:Covering_test}, instead of sampling a polygon $[\pi=x\to...\to x]$ \textbf{and} a starting position from $x'\in f^{-1}(x)$, they \textbf{only} sample a polygon. Then, they check that its lift to \textbf{every} vertex in the fiber $f^{-1}(x)$ is closed.

\begin{algorithm}[H]
\caption{\quad\texttt{Dinur-Meshulam covering tester}}\label{alg:DM_Covering_test}
\textbf{Input:} $f:\cG\rightarrow G(\cX),$ a finite covering of the $1$-skeleton of $\cX$. \\
\textbf{Output:} Accept or Reject.

\begin{algorithmic}[1]

\STATE Pick $[\pi]\in P(\cX)$ uniformly at random. 
\STATE Let $x$ be the starting vertex of some uniformly chosen representative $\pi$ of $[\pi]$. 
\STATE{If \textbf{for every} $x'\in f^{-1}(x)$ the lift $\pi'$ of $\pi$ to $x'$ is a closed path in $\cG$,
return Accept.}

\STATE {Otherwise, return Reject.}

\end{algorithmic}
\end{algorithm}

In Section \ref{sec:equivalence_covering_cocycles}, we give a one to one correspondence between $1$-cochains of a polygonal complex $\cX$ and coverings of its $1$-skeleton. To encapsulate the Dinur--Meshulam Algorithm, one can use this equivalence and the notations of Section \ref{sec:Cocycle_stability}, but with a different metric on $\Sym(n)$. They  use the discrete distance $$\forall \sigma,\tau \in \Sym(n)\ \colon\ \ d_D(\sigma,\tau)=\begin{cases}
    1 & \sigma\neq \tau,\\
    0& \sigma =\tau,
\end{cases}$$ 
while we are using the normalized Hamming distance on $\Sym(n)$.
This change may  seem minor, but it allows one to prove beautiful results in their setup (See Theorem 1.2 in  \cite{dikstein2023coboundary}) that seem very difficult to prove in our setting. The motivation to insist on our setting is coming from the fact that it can lead to a solution to Problem \ref{prob:sofic_groups}. For more details, see the open problem section of Part II of this paper \cite{CL_part2}.

There is another strengthening   of Algorithm \ref{alg:Covering_test} which is perpendicular to the Dinur--Meshulam one. Instead of sampling a single polygon with a single starting position (as done in Algorithm \ref{alg:Covering_test}), one can sample \textbf{for every} polygon $[\pi=x\to...\to x]\in P(\cX)$ a uniformly random  starting position  $x'_\pi\in f^{-1}(x)$. Then, the tester checks for each polygon and sampled starting position whether the lift of $\pi$ to $x'_\pi$ is closed or not. Namely:

\begin{algorithm}[H]
\caption{\quad\texttt{$L^\infty$ covering tester}}\label{alg:L_infty_Covering_test}
\textbf{Input:} $f:\cG\rightarrow G(\cX),$ a finite covering of the $1$-skeleton of $\cX$. \\
\textbf{Output:} Accept or Reject.

\begin{algorithmic}[1]

\STATE For every $[\pi]\in P(\cX)$ let $x_{\pi}$ be the starting vertex of a uniformly chosen representative $\pi$ of $[\pi]$. 
\STATE{Sample for every $[\pi]\in P(\cX)$ a lift $x'_\pi\in f^{-1}(x_\pi)$.}
\STATE {If for every $\pi$ and sampled $x'_\pi$, the lift  $\pi'$ of $\pi$ to $x'_\pi$ is  closed in $\cG$,
return Accept.}

\STATE {Otherwise, return Reject.}

\end{algorithmic}
\end{algorithm}

 When thinking of the \textbf{complex} $\cX$ as a constant, this tester is essentially equivalent to Algorithm \ref{alg:Covering_test}. To encapsulate the difference between Algorithm \ref{alg:L_infty_Covering_test} and Algorithm \ref{alg:Covering_test}, one can replace the expectations over polygons in our analysis with a maximum over polygons. The same change can be done to Algorithms \ref{alg:hom_test} and \ref{alg:Cocycle_test}, and we call these \emph{the $L^\infty$-analogues} of the testers.

\section{\textbf{Cocycle stability of polygonal complexes}} \label{sec:Cocycle_stability}
In this section, we define the last stability problem of interest (for this paper).
\subsection{Non-commutative cohomology}
Let $\cX$ be a polygonal complex. Recall our notation $\overrightarrow{\cX}(i)$ for the oriented  $i$-cells of $\cX$ and $\cX(i)$ for the unoriented versions, which were defined in  \eqref{eq:i-cells} and \eqref{eq:oriented_i-cells}. Moreover, for $i=1$ or $2$, recall that $\bar c$ is the reverse orientation of the cell $c\in \overrightarrow{\cX}(i)$. Let $\Gamma$ be a non trivial group and $d\colon \Gamma \times \Gamma \to \mathbb{R}_{\geq 0}$ a  bi-invariant metric on $\Gamma$.

  The \emph{$i$-cochains of $\cX$ with $\Gamma$  coefficients} are the anti-symmetric assignments of elements of $\Gamma$ to oriented $i$-cells. Namely,
    \[
        C^i(\cX,\Gamma)=\{\alpha\colon \overrightarrow{\cX}(i)\to \Gamma\mid \forall c\in \overrightarrow{\cX}(i)\colon \alpha(\bar c)=\alpha(c)^{-1}\}.
    \]
    Since we did not define orientations on vertices, there are no anti-symmetricity conditions on $0$-cochains. Also, for 
every $2$-cochain $\alpha\colon \raP(\cX)\to \Gamma$ and an orientation $\pi$ of a polygon $[\pi]$, we assume that $\alpha(\pi)$ is conjugate to $\alpha(\pi')$, where $\pi'$ is a shift of $\pi$.

The \emph{coboundary maps} are defined as follows. For a $0$-cochain $\alpha\colon V(\cX)\to \Gamma$, its coboundary $\delta\alpha$ is the $1$-cochain 
\[
\forall x\xrightarrow{e} y\in \raE(\cX)\ \colon\ \ \delta\alpha(e)=\alpha(x)^{-1}\alpha(y).
\]
For a $1$-cochain $\alpha\colon \raE(\cX)\to \Gamma$, its coboundary $\delta\alpha$ is the $2$-cochain
\[
\forall \pi=e_1...e_\ell\in \raP(\cX)\ \colon\ \ \delta\alpha(\pi)=\alpha(e_1)...\alpha(e_\ell).
\]
\begin{rem}
    Every map from $\raE(\cX)$ to a group $\Gamma$ can be extended to  paths in $G(\cX)$ in the natural way: For $\pi=x_0\xrightarrow{e_1}...\xrightarrow{e_\ell}x_\ell$, we let $\alpha(\pi)=\alpha(e_1)...\alpha(e_\ell)$. Note that $\delta \alpha(\pi)$ is exactly $\alpha(\pi)$ in the extended sense, but restricted to the polygons of $\cX$. 
\end{rem}
Next, we prove \emph{exactness}, namely, that for $\alpha\in C^0(\cX,\Gamma)$, $\delta^2\alpha$ is the constant identity function. For a polygon $\pi=x_0\xrightarrow{e_1}x_1\xrightarrow{e_2}...\xrightarrow{e_\ell}x_\ell=x_0$, we have
\[
\delta^2\alpha(\pi)=\delta\alpha(e_1)...\delta\alpha(e_\ell)=\alpha(x_0)^{-1}\alpha(x_1)\alpha(x_1)^{-1}\alpha(x_2)...\alpha(x_{\ell-1})^{-1}\alpha(x_\ell)=\alpha(x_0)^{-1}\alpha(x_\ell)=\Id.
\]
For any two $i$-cochains with $\Gamma$ coefficients $\alpha$ and $\beta$,  the \emph{distance} between them is 
\begin{equation}\label{eq:distance_between_cochains}
    d(\alpha,\beta)=\Ex_{[x]\in \cX(i)}\Ex_{x\in [x]}[d
(\alpha(x),\beta(x))].
\end{equation}
The \emph{norm} of an $i$-cochain with $\Gamma$ coefficients $\alpha$ is its distance to the constant identity cochain, namely 
\begin{equation}\label{eq:norm_of_cochain}
    \lVert\alpha\rVert=\Ex_{[x]\in \cX(i)}\Ex_{x\in[x]}[d
(\alpha(x),\Id)].\footnote{It is common to associate weights to the $i$-cells of the complex, and then to sample them in \eqref{eq:distance_between_cochains} and \eqref{eq:norm_of_cochain} according to their weight instead of uniformly at random. This is often used in the contemporary study of high dimensional expanders (cf. \cite{evra2016bounded}), and is related to Remark \ref{rem:general_distributions_hom_test}. We present this in more detail in Section \ref{sec:alt_dist}.}
\end{equation}

An \emph{$i$-cocycle} is an $i$-cochain $\alpha$ for which $\delta\alpha$ is the constant identity function. We denote the collection of $i$-cocycles by $Z^i(\cX,\Gamma)$.
A $0$-cochain $\beta\colon V(\cX)\to \Gamma$ is said to be a \emph{$0$-coboundary} if it is constant. Namely, for every $x,y\in V(\cX)$ we have $\beta(x)=\beta(y)$. A $1$-cochain $\alpha\colon \raE(\cX)\to \Gamma$ is said to be a \emph{$1$-coboundary} if it is in the image of the coboundary operator $\delta\colon C^0(\cX,\Gamma)\to C^1(\cX,\Gamma)$. 
Namely, there exists a $0$-cochain $\beta\colon V(\cX)\to \Gamma$ such that for every $xy\in \raE(\cX)$, $\alpha(xy)=\delta\beta(xy)=\beta(x)^{-1}\beta(y)$. We denote by $B^i(\cX,\Gamma)$ the collection of $i$-coboundaries of $\cX$.  Note that, without assuming further assumptions on $\Gamma$, the only indices for which we  define $Z^i(\cX,\Gamma)$ and $B^i(\cX,\Gamma)$ are $i=0\ \textrm{or}\ 1$. 
 We say that the \emph{$i^{\rm th}$ cohomology of $\cX$ with $\Gamma$ coefficients vanishes} if
 every $i$-cocycle of $\cX$  is an $i$-coboundary. 

 \subsection{Cohomology with permutation coefficients}
Though all the cohomological definitions we listed were  defined for a general group $\Gamma$, in this paper we focus on the case where $\Gamma$ is a finite permutation group equipped with the normalized Hamming distance \eqref{eq:permutation_normalized_Hamming_distance}.\footnote{As mentioned in Section \ref{sec:Dinur-Meshulam_framework}, to get the Dinur--Meshulam framework you replace the Hamming metric with the discrete metric.} Moreover, since the normalized Hamming distance (with errors) can compare permutations of different sizes, we will study them in a collective manner. Let 
\[
C^i(\cX,\Sym)=\bigsqcup_{n=2}^{\infty} C^i(\cX,\Sym(n))\  ,\quad Z^i(\cX,\Sym)=\bigsqcup_{n=2}^{\infty} Z^i(\cX,\Sym(n))\ ,
\]
and 
\[
B^i(\cX,\Sym)=\bigsqcup_{n=2}^{\infty} B^i(\cX,\Sym(n))
\]
be the \emph{$i$-cochains with permutation coefficients}, the \emph{$i$-cocycles with permutation coefficients} and the \emph{$i$-coboundaries with permutation coefficients}, respectively. Thus, $d_h$ defines a metric on $C^i(\cX,\Sym)$ as in  \eqref{eq:distance_between_cochains}, and thus also a norm on it as in equation \eqref{eq:norm_of_cochain}.

As the title of the section suggests, we now devise a tester for deciding whether a given $i$-cochain with permutaition coeffcients is an $i$-cocycle.   
Note that, by definition, $\alpha$ is a cocycle if and only if $\Vert\delta\alpha\Vert=0$. How robust is this property? Namely, if $\delta\alpha$ has a small norm, does it imply that $\alpha$ is close to a cocycle? The following tester  assumes exactly that:

\begin{algorithm}[H]
\caption{\quad\texttt{$i$-cocycle tester, for $i=0$ or $1$}}\label{alg:Cocycle_test}
\textbf{Input:} $\alpha\colon \overrightarrow{\cX}(i)\to \Sym(n)$ an $i$-cochain.\\
\textbf{Output:} Accept or Reject.

\begin{algorithmic}[1]

\STATE Pick an un-oriented $i$-cell $[\pi]\in \cX(i+1)$ uniformly at random. Then, pick a representative $\pi$ of $[\pi]$ uniformly at random. 
\STATE Pick an index $i\in \{1,...,n\}$ uniformly at random. 
\STATE{If $\delta\alpha(\pi).i=i$, then
return Accept.}

\STATE {Otherwise, return Reject.}

\end{algorithmic}
\end{algorithm}

Algorithm \ref{alg:Cocycle_test} can be framed in the setup of Section \ref{sec:Prop_testing} in a similar way to Algorithm \ref{alg:hom_test}. Namely, the input function is of the form $\alpha\colon \cX(i)\times [n]\to [n]$ with the restriction that for every $i$-cell $e$, $\alpha(e,\cdot)\colon [n]\to[n]$ is a permutation. Hence:
\begin{itemize}
    \item $\frF$ is the $i$-cochains $C^i(\cX,\Sym)$;
    \item $\frP$ is the $i$-cocycles $Z^i(\cX,\Sym)$;
    \item The normalized Hamming distance between two $i$-cochains $\alpha,\beta\in C^i(\cX,\Sym)$ is
    \[
        d_h(\alpha,\beta)=\Ex_{[e]\in \cX(i)}[d_h(\alpha(e),\beta(e))].
    \]
    Hence, the \emph{cocycle global defect} of $\alpha$ is
    \[
    \begin{split}
        \Defect_{cocyc}(\alpha)&=d_h(\alpha,Z^i(\cX,\Sym))\\
        &=\inf\{d_h(\alpha,\varphi)\mid N\in \mathbb{N},\ \varphi\colon \cX(i)\to \Sym(N),\ \Vert\delta\varphi\Vert=0\}.
    \end{split}
\]
\item The rejection probability of $\alpha$ in the  cocycle tester is, by design, $\Vert\delta\alpha\Vert$. Hence, the \emph{cocycle local defect} of $\alpha$ is 
\[
\defect_{cocyc}(\alpha)=||\delta\alpha||.
\]
    
\end{itemize}
\begin{rem}
    Our notations are suggestive of the $i=1$ case, since this case is our main interest. 
\end{rem}
\begin{defn}
Let $\rho$ be a rate function, as in Definition \ref{defn:Rate_function}.
    A polygonal complex is said to be  \emph{$\rho$-cocycle stable in the $i^{\rm th}$ dimension} (where $i=0$ or $1$) if for every $i$-cochain $\alpha$ we have
    \[
        d_h(\alpha,Z^i(\cX,\Sym))\leq \rho(\Vert\delta\alpha\Vert).
    \]
\end{defn}  

The following subsections are motivational. First, in Section \ref{sec:cocyc_stab_0_dim} we study cocycle stability in the $0^{\rm th}$ dimension, and prove it is related to the spectral gap of the underlying graph. This hints at a Cheeger-type inequality. Following the hint, in Section \ref{sec:Cheeger_cst}  we define an appropriate notion of a Cheeger constant, and compare it with the classical Cheeger constant. Finally, in Section \ref{sec:coboundary_cosystolic} we discuss coboundary  expansion  to motivate generalizations of it to our setup. This generalization will be further studied in Part II of this paper \cite{CL_part2}, and in the companion paper \cite{CL_stability_Random_complexes}.

\subsection{Cocycle stability in the $0^{\rm th}$ dimension}\label{sec:cocyc_stab_0_dim}
Cocycle stability in the $0^{\rm th}$ dimension  depends only on the underlying graph of the complex. For every finite graph $\cX$ there is an associated adjacency matrix $M_\cX\colon V(\cX)\times V(\cX)\to \mathbb{C} $ that satisfies
\begin{equation}\label{eq:adjacency_matrix}
    M_\cX(x,y)=|\{e\in \raE(\cX)\mid \iota(e)=x,\tau(e)=y\}|.
\end{equation}
Because of the reverse edge function, this matrix is symmetric. Assuming the graph is $k$-regular, the eigenvalues of this matrix are 
\[
k=\lambda_1\geq \lambda_2\geq ...\ge \lambda_{|V(\cX)|}\geq -k.
\]
The normalized spectral gap of a $k$-regular graph $\cX$ is
\[
\gamma(\cX)=\frac{k-\lambda_2}{k}.
\]
\begin{lem}[$L^2$-Poincare-type inequality]\label{lem:Poincare_ineq}
    Let $\cX$ be a finite regular graph. Let $\mathbb{H}$ be a Hilbert space. Let $f\colon V(\cX)\to \mathbb{H}$ be any function. Then 
    \[
        \Ex_{xy\in \raE(\cX)}[\Vert f(x)-f(y)\Vert^2]\geq \gamma(\cX) \cdot\Ex_{x,y\in V(\cX)}[\Vert f(x)-f(y)\Vert^2].
    \]
\end{lem}

\begin{proof}
    See Theorem 13.9 in \cite{Hoory_Linial_Wigderson} for a proof sketch, and Section 1.3 in \cite{de_la_Salle_spectral_gap} for a complete proof.
\end{proof}

\begin{prop}\label{prop:Cheeger_lower_bound}
Every finite connected regular graph $\cX$ is $\rho$-cocycle stable in the $0^{\rm th}$ dimension, with $$\rho(\eps)=\frac{\eps}{\gamma(\cX)}.$$
\end{prop}

\begin{proof}
    First of all, since $\cX$ is connected,  all $0$-cocycles are $0$-coboundaries, namely  constant functions. This is because $\Id=\delta\alpha(xy)=\alpha(x)^{-1}\alpha(y)$ implies $\alpha(x)=\alpha(y)$ along every edge.

    Now, let $\alpha\colon V(\cX)\to \Sym(n)$ be a $0$-cochain. The group $\Sym(n)$ embeds naturally into $U(n)$, the $n\times n$ unitary matrices, as permutation matrices. Moreover, for two permutations $\sigma,\tau\in \Sym(n)$, we have
    \[
        d_h(\sigma,\tau)=\frac{1}{2}\Vert \sigma -\tau \Vert_{HS}^2,
    \]
    where on the right hand side we think of $\sigma$ and $\tau$ as the associated permutation matrices, and 
    \begin{equation}
        \forall A\in M_{n}(\mathbb{C})\ \colon\ \ \Vert A\Vert_{HS}^2=\frac{\tr (A^*A)}{n}=\frac{\sum_{i,j=1}^n |A_{ij}|^2}{n}
    \end{equation} 
    is the normalized Hilbert--Schmidt norm, which is induced by an inner product on $M_n(\mathbb{C})$ that makes it a Hilbert space.
    Hence, by Lemma \ref{lem:Poincare_ineq},
    \[
       \begin{split}        \Vert\delta\alpha\Vert&=\Ex_{xy\in  \raE(\cX)}[d_h(\delta\alpha(xy),\Id)]\\
       &=\Ex_{xy\in  \raE(\cX)}[d_h(\alpha(x)^{-1}\alpha(y),\Id)]\\
        &=\Ex_{xy\in  \raE(\cX)}[d_h(\alpha(y),\alpha(x)]\\
         &=\frac{1}{2}\Ex_{xy\in  \raE(\cX)}[\Vert \alpha(y)-\alpha(x)\Vert_{HS}^2]\\
         &\geq\frac{\gamma(\cX)}{2}\Ex_{x,y\in  V(\cX)}[\Vert \alpha(y)-\alpha(x)\Vert_{HS}^2]\\
         &=\gamma(\cX)\Ex_{x,y\in V(\cX)}[d_h(\alpha(y),\alpha(x))]\\
         &\ge \gamma(\cX) \min_{x\in V(\cX)}\underbrace{\Ex_{y\in V(\cX)}[d_h(\alpha(y),\alpha(x))]}_{d_h(\alpha,\alpha(x))}\\
         &\geq \gamma(\cX) d_h(\alpha,Z^0(\cX,\Sym)),
       \end{split} 
    \]
    proving the claim.
\end{proof}
\subsection{Cheeger constants and (high dimensional) expansion}\label{sec:Cheeger_cst}
As was probably evident to an observant reader, Proposition \ref{prop:Cheeger_lower_bound} is a type of \emph{Cheeger inequality}. To make this statement precise, we need a suitable notion of a Cheeger constant.

\begin{defn}\label{defn:gen_cheeger}
    The  $i^{\rm th}$ \emph{cocycle Cheeger constant} of a complex $\cX$ with $\Gamma$ coefficients, where $i=0$ or $1$, is 
    \begin{equation}\label{eq:Cheeger_constant}
    h_i(\cX,\Gamma)=\inf\left\{\frac{\Vert\delta\alpha\Vert}{d(\alpha,Z^i(\cX,\Gamma))}\ \middle\vert\ {\alpha \in C^i(\cX,\Gamma)},\ \Vert\delta\alpha\Vert\neq 0\right\}.
    \end{equation}
    Similarly, the  $i^{\rm th}$ \emph{coboundary Cheeger constant} of a complex $\cX$ with $\Gamma$ coefficients is 
    \begin{equation}\label{eq:Cheeger_constant}
    h^B_i(\cX,\Gamma)=\inf\left\{\frac{\Vert\delta\alpha\Vert}{d(\alpha,B^i(\cX,\Gamma))}\ \middle\vert\ {\alpha \in C^i(\cX,\Gamma)} \setminus B^i(\cX,\Gamma)\right\}.
    \end{equation}
    When $h_i(\cX,\Sym)>0$ we call $\cX$ an \emph{$i$-cocycle expander with $\Gamma$ coefficients}, and when $h^B_i(\cX,\Sym)>0$ we call it an \emph{$i$-coboundary expander with $\Gamma$ coefficients}.
    
\end{defn}

When taking permutation coefficients, these notions genralize the classical Cheeger constant for graphs when $i=0$, as explained in Remark \ref{rem:classical_cheeger}. Moreover, when $i=1$ and $\Gamma=\FF_2$, these notions are closely related to Gromov's topological overlapping property \cites{gromov2010singularities,DKW_cosystolic_exp_implies_overlap}. As we aim to develop  further study of cocycle stability in permutations, we suggest in  \cite{CL_part2} a generalization of the search for bounded coboundary expanders.

\begin{rem}\label{rem:classical_cheeger}
    Let $\cX$ be a finite connected  $k$-regular graph. The \emph{classical Cheeger constant} of  $\cX$ is 
    \[
        h(\cX)=\min\left\{ \frac{|E(A,\bar A)|}{\min(|A|,|\bar A|)}\ \middle\vert \ \emptyset\neq A\subsetneq V(\cX)\right\},
    \]
    where $E(A,\bar A)$ is the set of edges with one endpoint in $A$ and the other in its complement $\bar A$. Note that if $\alpha\colon V(\cX)\to \{0,1\}$ is the characteristic function of $A$, and we identify $\{0,1\}$ with $\Sym(2)$ (as we did in Section \ref{sec:LTC}), then $\frac{|E(A,\bar A)|}{|E(\cX)|}=\Vert\delta\alpha\Vert$. Moreover, the $0$-cocycles with $\Sym(2)$ coefficients are exactly the characteristic functions of $V(\cX)$ and $\emptyset$, which we denote by ${\bf 1}$ and ${\bf 0}$ for the moment. Now, $d_h(\alpha,{\bf 0})=\Vert \alpha \Vert=\frac{|A|}{|V(\cX)|}$ and $d_h(\alpha,{\bf 1})=1-\Vert \alpha \Vert=\frac{|\bar A|}{|V(\cX)|}$. It can be verified that the closest constant function to $\alpha$ is either ${\bf 0}$ or ${\bf 1}$, even if we allow to range over all constant functions into $\Sym(n)$. Hence, $d_h(\alpha,Z^0(\cX,\Sym))=\frac{\min(|A|,|\bar A|)}{|V(\cX)|}$. All in all, 
    \begin{equation}\label{eq:rel_between_classical_and_0_Cheeger}
   \begin{split}
            h(\cX)&=\min\left\{ \frac{|E(A,\bar A)|}{\min(|A|,|\bar A|)}\ \middle\vert \ \emptyset\neq A\subsetneq V(\cX)\right\}\\
            &=\frac{|E(\cX)|}{|V(\cX)|}\cdot\min\left\{ \frac{\nicefrac{|E(A,\bar A)|}{|E(\cX)|}}{\nicefrac{\min(|A|,|\bar A|)}{|V(\cX)|}}\ \middle\vert \ \emptyset\neq A\subsetneq V(\cX)\right\}\\
            &=\frac{k}{2}\cdot\min\left\{ \frac{\Vert\delta \alpha\Vert}{d_h(\alpha,Z^0(\cX,\Sym))}\ \middle\vert \ \alpha\in C^0(\cX,\Sym(2)),\Vert\delta\alpha\Vert\neq 0\right\}\\
            &=\frac{k}{2}h_0(\cX,\FF_2)\\
            &\geq \frac{k}{2}h_0(\cX,\Sym).
        \end{split}     
    \end{equation}
    Namely, the classical Cheeger constant is (up to a factor of $\nicefrac{k}{2}$) the $0^{\rm th}$ cocycle Cheeger constant with $\FF_2\cong \Sym(2)$ coefficients, which upper bounds the cocycle Cheeger constant with permutation coefficients. 
  By the proof of Proposition \ref{prop:Cheeger_lower_bound}, 
  \begin{equation}\label{eq:Cheeger_ineq_lower}
      h_0(\cX,\Sym)\geq \gamma(\cX).
  \end{equation}
  Combining \eqref{eq:rel_between_classical_and_0_Cheeger} and \eqref{eq:Cheeger_ineq_lower}, we derive the lower bound in the Cheeger inequalities \cite{ALON_Milman_Cheeger_1,Alon_Cheeger_2,Dodziuk_Cheeger} (See also Proposition 
 4.2.5 in \cite{lubotzky1994discrete}). Moreover, since the Cheeger inequalities deduce an upper bound on $h(\cX)$ as a function of the spectral gap, together with \eqref{eq:rel_between_classical_and_0_Cheeger} we can deduce that also $h_0(\cX,\Sym)$ is upper bounded by the same bound.
\end{rem}   

\begin{cor}[Generalized  Cheeger inequalities in the $0^{\rm th}$ dimension]
    Let $\cX$ be a $k$-regular connected graph. Then,
\[
\gamma (\cX) \leq h_0(\cX,\Sym)\leq \sqrt{8\gamma(\cX)}.
\]

\end{cor}

\subsection{Coboundary expansion} \label{sec:coboundary_cosystolic}
In the beginning of the proof of Proposition \ref{prop:Cheeger_lower_bound}, we  deduced the following:

 \begin{fact}\label{fact:0_vanishing=connected}
     The $0^{\rm th}$ cohomology of $\cX$ with 
 $\Gamma$ coefficients vanishes if and only if $G(\cX)$ is connected.
 \end{fact} 
 
Hence, when we restricted the discussion in Proposition \ref{prop:Cheeger_lower_bound} and Remark \ref{rem:classical_cheeger} to the connected graph case, we forced the vanishing of the $0^{\rm th}$ cohomology. In that case, as proved in Remark \ref{rem:classical_cheeger}, the cocycle Cheeger constant with $\FF_2$ coefficients is the same as the classical Cheeger constant. But, for disconnected graphs they differ, and the natural analogue is the coboundary Cheeger constant. The following is a straightforward observation:

\begin{fact}\label{fact:positive_coboundary_and_vanishing}
    Let $\cX$ be a polygonal complex. Then:
    \begin{itemize}
        \item If $h^B_i(\cX,\Gamma)>0$, then, in particular, the $i^{\rm th}$ cohomology of $\cX$ with $\Gamma$ coefficients vanishes.
        \item If the $i^{\rm th}$ cohomology of $\cX$ with $\Gamma$ coefficients vanishes, then $h^B_i(\cX,\Gamma)=h_i(\cX,\Gamma)$.
    \end{itemize}
\end{fact}

The first observation in Fact \ref{fact:positive_coboundary_and_vanishing} (for $\FF_2$ coefficients) was a key ingredient in Linial--Meshulam's work on cohomology vanishing of random simplicial complexes \cite{linial_meshulam2006homological}. In  a companion paper to this one \cite{CL_stability_Random_complexes}, we discuss it further, and suggest to study phase transition of the vanishing of the $1^{\rm st}$ cohomology with permutation coefficients. But, what does it mean for the $1^{\rm st}$ cohomology with permutation coefficients to vanish? The analysis in Section \ref{sec:equivalences} provides a characterization:
\begin{fact}\label{fact:vanishing_of_1_cohomology}
    Let $\cX$ be a connected polygonal complex. Then, the $1^{\rm st}$ cohomology of $\cX$ with permutation coefficients vanishes if and only if the fundamental group $\pi_1(\cX,*)$ has no proper finite index subgroups. Equivalently, $\widehat{\pi_1(\cX,*)}=1$, i.e., the profinite completion of the fundamental group is trivial.
\end{fact}

As oppose to the $0^{\rm th}$ dimensional case, where vanishing of the $0^{\rm th}$ cohomology, namely connectedness, is the most basic property you would expect from an object which presumes to \emph{expand}, this is not the case in  higher dimensions. Out of the various suggested high dimensional expansion properties (cf. \cite{lubotzky2018high}), coboundary expansion is considered a very strong notion.
The motivation to study coboundary expansion over $\FF_2$ arised from the notion of topological overlapping \cite{gromov2010singularities}, but it also found applications in property testing \cite{kaufman2014high}. As mentioned before, we further discuss these notions, as well as topological overlapping, in Part II \cite{CL_part2} and in \cite{CL_stability_Random_complexes}, where we will also illustrate further the importance of coboundary expansion with permutation coefficients.

\section{\textbf{Proving equivalences}}\label{sec:equivalences}
\subsection{Covering stability is 
equivalent to cocycle stability} \label{sec:equivalence_covering_cocycles}
In this (sub) section, we prove  Theorem \ref{thm:main1}. Namely, we show that $\rho$-cocycle stability and $\rho$-covering stability of a polygonal complex $\cX$ are the same. We actually define a one to one correspondence between  $1$-cochains of $\cX$ with $\Sym$ coefficients (up to some equivalence relation) and coverings of the $1$-skeleton of $\cX$ (up to isomorphism). Under this correspondence, $1$-cocycles are mapped to genuine coverings. Moreover, the global defects are in correspondence. Lastly, the testers translate to one another under this correspondence, which in turn preserves the local defect. This idea of encoding graph coverings using $1$-cochains is classical (cf. \cites{Hatcher_Alg_Top,Linial_Amit2002random,amit2001random}), and was specifically used in Dinur--Meshulam \cite{Dinur-Meshulam} in a similar way.

To make the claims from the previous paragraph precise, we need the following action of $0$-cochains of $\cX$ with $\Sym(n)$ coefficients on $1$-cochains of $\cX$ with $\Sym(n)$ coefficients:
\begin{equation}\label{eq:action_0-coch_on_1-coch}
    \forall \alpha \colon \raE(\cX)\to \Sym(n),\ \beta \colon V(\cX)\to \Sym(n)\ \colon\ \ \beta.\alpha(x\xrightarrow{e}y)=\beta(x)^{-1}\alpha(e)\beta(y).
\end{equation}
Note that for every polygon $\pi=x\xrightarrow{e_1}...\xrightarrow{e_\ell}x\in \raP(\cX)$, we have $\beta.\alpha(\pi)=\beta(x)^{-1}\alpha(\pi)\beta(x)$. In particular, $\Vert\delta(\beta.\alpha)\Vert=\Vert \delta \alpha\Vert$, and hence the action preserves $1$-cocycles.
\begin{claim}\label{claim:encoding_of_covers_by_1-cocyc}
    Let $\cX$ be a polygonal complex. There is a one to one correspondence between coverings of $G(\cX)$ and orbits of $1$-cochains of $\cX$ with permutation coefficients under the action described in \eqref{eq:action_0-coch_on_1-coch}.
\end{claim}

\begin{proof}
    We first associate with every $1$-cochain $\alpha\colon \raE(\cX)\to \Sym(n)$ a corresponding covering $f\colon \cG \to G(\cX)$ as follows:
    
 \[
 \begin{split}
      &V(\cG)=V(\cX)\times [n]\quad,\quad\raE(\cG)=\raE(\cX)\times [n];\\
      &\forall x\xrightarrow{e}y\in \raE(\cX),i\in [n]\ \colon \ \ \tau(e,i)=(y,i),\ \iota(e,i)=(x,\alpha(e).i),\ \overline{(e,i)}=(\bar e,\alpha(e).i);\\
      &\forall x\in V(\cX),e\in \raE(\cX),i\in[n]\ \colon \ \ f(x,i)=x,\quad f(e,i)=e.
 \end{split}
 \]
By Fact \ref{fact:char_of_graph_coverings}, $f$ is a covering.\footnote{Note that the permutation $\alpha(e)$ tells us how the fiber over the terminal point $y$ is mapped to the fiber over the origin point $x$ and not the other way around. This is because of our choice of left actions.}

On the other hand, if $f\colon \cG\to G(\cX)$ is a degree $n$ covering, then one can construct a $1$-cochain $\alpha\colon \raE(\cX)\to \Sym(n)$ as follows:
     For every $x\in V(\cX)$, $|f^{-1}(x)|=n$. Hence we can label the vertices of $f^{-1}(x)$ by $\{(x,i)\}_{i=1}^n$. Note that for each vertex there are $n!$ ways of choosing these labels. Now, for every $e\in \raE(\cX)$ and $e'\in f^{-1}(e)$ define $e'=(e,i)$ if $\tau(e')=(y,i)$. Here we did not have any choice in the labeling. Now, define $\alpha(e).i$ to be  the second coordinate of $\iota(e,i)$. Note that  a different choice of labeling for the fibers $f^{-1}(x)$ would give rise to another $1$-cochain in the same orbit of the action of the $0$-cochains. This proves the correspondence.
\end{proof}

The encoding in Claim \ref{claim:encoding_of_covers_by_1-cocyc} allows us to analyze the lifted path $\pi'$ used in Algorithm \ref{alg:Covering_test}. Recall that for every path $x\xrightarrow{e_1}...\xrightarrow{e_\ell}y$ in $\cX$ and every $x'\in f^{-1}(x)$, there exists a unique path $\pi'$ in $\cG$ that starts at $x'$ and satisfies $f(\pi')=\pi$. By choosing a $1$-cochain $\alpha$ that corresponds to the covering $f$, we have chosen a label $(x,i)$ for $x'$. To lift $x\xrightarrow{e_1}x_1$, we need to find the index $j$ such that $(x,i)\xrightarrow{(e_1,j)}(x_1,j)$. By our construction, $i=\alpha(e_1).j$, which implies $j=\alpha(e_1)^{-1}.i=\alpha(\bar e_1).i$. Therefore, by induction, the lifted path $\pi'$ is of the form
 \[
 \begin{split}
\pi'=&(x,i)\xrightarrow{(e_1,\alpha(\bar e_1).i)}(x_1,\alpha(\bar e_1).i)\xrightarrow{(e_2,\alpha(\bar e_2)\circ\alpha(\bar e_1).i)}(x_2,\alpha(\bar e_2)\circ\alpha(\bar e_1).i)\to\dots\\
&\dots\xrightarrow{(e_\ell,\alpha(\bar e_\ell)\circ...\circ\alpha(\bar e_1).i)}  (y,\alpha(\bar e_\ell)\circ...\circ\alpha(\bar e_1).i).
\end{split}
 \]
Note that $\alpha(\bar e_\ell)\circ...\circ\alpha(\bar e_1)=\delta\alpha(\pi)^{-1}$. Hence, the lifted path of a polygon  $\pi$ to $x'=(x,i)$ is closed if and only if $\delta\alpha(\pi)^{-1}.i=i$ (which is if and only if $\delta\alpha(\pi).i=i$). Therefore, the testers Algorithm \ref{alg:Cocycle_test} and Algorithm \ref{alg:Covering_test} are \textbf{the same} when encoding coverings as $1$-cochains. In particular, the covering $f$ is genuine if and only if $\delta\alpha(\pi)=\Id$ for every polygon $\pi\in \raP(\cX)$, which is the same as saying $\delta\alpha=\Id$. 
\begin{cor}\label{cor:local_defect_preserved}
    Let $\cX$ be a polygonal complex.  The \textbf{local} defects of  Algorithms \ref{alg:Covering_test} and \ref{alg:Cocycle_test} are preserved under the corresponednce of Claim \ref{claim:encoding_of_covers_by_1-cocyc}. In particular, orbits of $1$-cocycles of $\cX$ are in one to one correspondence with isomorphism classes of genuine coverings of $\cX$.
\end{cor}

\begin{claim}\label{claim:global_def_preserved}
    Let $\cX$ be a polygonal complex.  The \textbf{global} defects of  Algorithms \ref{alg:Covering_test} and \ref{alg:Cocycle_test} are preserved under the corresponednce of Claim \ref{claim:encoding_of_covers_by_1-cocyc}. 
\end{claim}
\begin{proof}
 Let $\alpha\colon \raE(\cX)\to \Sym(n)$ and $\beta\colon \raE(\cX)\to \Sym(N)$ be two $1$-cochains, where $N\geq n$. Let $f\colon \cG\to G(\cX)$ and $\varphi\colon \cY \to G(\cX)$ be the coverings associated with $\alpha$ and $\beta$ respectively, as defined in the proof of Claim \ref{claim:encoding_of_covers_by_1-cocyc}. Then, we can define a mutual $G(\cX)$-labeled subgraph $\cA$ of $\cG$ and $\cY$ as follows: Its vertices are $$V(\cA)=V(\cG)=V(\cX)\times [n].$$ Its edges are $$\raE(\cA)=\raE(\cG)\cap\raE(\cY)=\{(e,i)\mid \iota_\cG(e,i)=\iota_\cY(e,i)\}=\{(e,i)\mid \alpha(e).i=\beta(e).i \}.$$
Therefore, we found a $G(\cX)$-labeled subgraph of $\cG$ and $\cY$ with $$|\raE(\cA)|=|\{(e,i)\mid e\in \raE(\cX),\ i\in [N],\ \alpha(e).i=\beta(e).i\}|=|\raE(\cX)|\cdot N\cdot d_h(\alpha,\beta).$$
Hence, recalling the definition of the normalized edit distance \eqref{eq:norm_edit_distance},
\begin{equation}\label{proof:edit_smaller_than_hamming}
    d_{edit}(\cG,\cY)\leq 1-\frac{|\raE(\cA)|}{|\raE(\cY)|}=d_h(\alpha,\beta).
\end{equation}

 On the other hand, let $f\colon \cG\to G(\cX)$ be an $n$-covering and $\varphi \colon \cY\to G(\cX)$ be an $N$-covering, where $N\geq n$.  If $a\colon \cA\to G(\cX)$ is the largest $G(\cX)$-labeled graph that embeds in both, then we can think of $V(\cA)\subseteq V(\cG)\cap V(\cY)$. For every $x\in V(\cX)$, if $|a^{-1}(x)|=k_x\leq n$ then we can label its vertices as $\{(x,i)\}_{i=1}^{k_x}$. Further, label the vertices of $f^{-1}(x)\setminus a^{-1}(x)$ by $\{(x,i)\}_{i=k_x+1}^{n}$ and the vertices of $\varphi^{-1}(x)\setminus a^{-1}(x)$ by $\{(x,i)\}_{i=k_x+1}^{N}$. Note that we essentially chose an arbitrary embedding of $V(\cG)\setminus V(\cA)$ into $V(\cY)\setminus V(\cA)$. By this choice of labeling, we get $\raE(\cA)\subseteq \raE(\cG)\cap\raE(\cY)$. But $\cA$ was the largest $G(\cX)$-labeled embedding into $\cG$ and $\cY$, which implies  $\raE(\cA)= \raE(\cG)\cap\raE(\cY)$. Let $\alpha$ and $\beta$ be the  $1$-cochain encodings of $f$ and $\varphi$ respectively, defined by our choice of labelings (as done in the proof of Claim \ref{claim:encoding_of_covers_by_1-cocyc}). By definition, 
 \[
 \begin{split}
     \raE(\cG)\cap\raE(\cY)&=\{(e,i)\mid e\in \cX,\ i\in[N],\ \iota_\cG(e,i)=\iota_\cY(e,i)\}\\
     &=\{(e,i)\mid e\in \cX,\ i\in[N],\ \alpha(e).i=\beta(e).i\}.
 \end{split}
 \]
 Therefore,
 \begin{equation}
     d_h(\alpha,\beta)=\Pro_{e\in \raE(\cX),i\in[N]}[\alpha(e).i=\beta(e).i]=\frac{|\raE(\cG)\cap\raE(\cY)|}{|\raE(\cY)|}=d_{edit}(\cG,\cY).
 \end{equation}
 Thus, if $\varphi\colon \cY\to \cX$ is the closest genuine cover to $f\colon \cG\to G(\cX)$, then there is an encoding of them where the normalized Hamming distance is the same as the normalized edit distance. Together with \eqref{proof:edit_smaller_than_hamming}, this proves that the global defects are preserved under this correspondence.
\end{proof}

By combining Claim \ref{claim:encoding_of_covers_by_1-cocyc}, Corollary \ref{cor:local_defect_preserved} and Claim \ref{claim:global_def_preserved}, Theorem \ref{thm:main1} is proven.

\subsection{Homomorphism stability is equivalent to  cocycle stability}\label{sec:hom_equiv_to_cocycle}
In this section we prove Theorems \ref{thm:main2} and \ref{thm:main3}.
Like in the previous section, there is a  classical correspondence between presentations and polygonal complexes. One direction, from presentations to complexes, is via the \emph{presentation complex} of the group (cf. Corollary 1.28, page 52, in \cite{Hatcher_Alg_Top}). 
The other direction, from complexes to presentaions, is via the presentation of the fundamental group associated with the reatraction of a spanning tree. We now recall these constructions in detail.

\begin{defn}\label{defn:presentation_complex}
    Given a presentation $\Gamma =\langle S|R\rangle$, one can define the presentation complex $\cX_{\langle S|R\rangle}$ to be the following polygonal complex: It has a single vertex $*$, and an edge $e(s)$ for every generator $s\in S$. Then, for every $r= s_1^{\eps_1}\cdot...\cdot s_\ell^{\eps_\ell}\in R$, we add a polygon (together with all its orientations) $\pi(r)=e(s_1)^{\eps_1}...e(s_\ell)^{\eps_\ell}$, where $e(s)^{-1}=\overline{e(s)}$ and $e(s)^{1}=e(s)$. Note that since there is only one vertex, this is a closed path regardless of $r$. 
\end{defn}

\begin{proof} [Proof of Theorem \ref{thm:main2}]
    Every function $f\colon S\to\Sym(n)$ can be extended to a $1$-cochain  $\alpha_f$ on $\cX_{\langle S|R\rangle}$ in an antisymmetric way. Namely, $\alpha_f(e(s))=f(s),\alpha_f(\overline{e(s)})=f(s)^{-1}$. Furthermore, for $r= s_1^{\eps_1}\cdot...\cdot s_\ell^{\eps_\ell}\in R$ we get $$\delta\alpha_f(\pi(r))=\prod_{i=1}^{\ell}\alpha_f(e(s_i))^{\eps_i}=\prod_{i=1}^{\ell}f(s_i)^{\eps_i}=f(r).$$  Hence
    \[
    \begin{split}
        \defect_{cocyc}(\alpha_f)&=\Ex_{[\pi]\in P(\cX_{\langle S|R\rangle})}[d_h(\delta\alpha_f(\pi),\Id)]\\
        &=\Ex_{r\in R}[d_h(\delta\alpha_f(\pi(r)),\Id)]\\
        &=\Ex_{r\in R}[d_h(f(r),\Id)]\\
        &=\defect_{hom}(f).
    \end{split}
    \]
    In particular, this correspondence $f\leftrightarrow \alpha_f$ associates the homomorphisms $\Hom(\Gamma,\Sym)$  with the cocycles $Z^1(\cX_{\langle S|R\rangle},\Sym)$.
    
Let $\varphi\colon S\to \Sym(N)$ be another function, and let $\alpha_\varphi$ be the associated $1$-cochain on $\cX_{\langle S|R\rangle}$. Then,
\[
\begin{split}
    d_h(\alpha_f,\alpha_\varphi)&=\Ex_{e\in \raE(\cX_{\langle S|R\rangle})}[d_h(\alpha_f(e),\alpha_\varphi(e))]\\
    &=\Ex_{s\in S}[d_h(\alpha_f(e(s)),\alpha_\varphi(e(s)))]\\
    &=\Ex_{s\in S}[d_h(f(s),\varphi(s))]\\
    &=d_h(f,\varphi).
\end{split}
\]
Therefore, the global defects are preserved. 
\end{proof}

We now recall the presentation of the fundamental group associated with the retraction of a tree. This presentation can be attained by van Kampen's Theorem (See Chapter 1.2 in \cite{Hatcher_Alg_Top}).
\begin{defn}\label{defn:pres_of_fundamental_group}
    Let $\cX$ be a connected polygonal complex.  Choose a base point $*\in V(\cX)$  and a spanning tree $T$ in $\cX$. There is a standard presentation $\langle S|R\rangle$ of the fundamental group $\pi_1(\cX,*)$  with respect to $T$.
 For the generator set $S$, we choose a representative $e$ out of every unoriented edge $\{e,\bar e\}\in E(\cX)$ which is outside of the spanning tree $T$. 
For relations, we take a single  representative $\pi$ from each unoriented polygon $[\pi]\in P(\cX)$. Now, the way we interpret $\pi$ as a word in $S$ is by replacing each $\bar e$ for which $e\in S$ by $e^{-1}$, and by deleting every edge $e\in T$. 
\end{defn}

\begin{proof} [Proof of Theorem \ref{thm:main3}]
     Let $\alpha\colon  \raE(\cX)\to \Sym(n)$ be a $1$-cochain of $\cX$. We now define a $0$-cochain $\beta$ which depends on $\alpha,*$ and $T$:
For every vertex $y\in V(\cX)$, let $\pi_y$ be the unique path in $T$ from $y$ to the base point $*$. Define $\beta\in C^0(\cX,\Sym)$ by $$\beta(y)=\alpha(\pi_y).$$ Let $y\xrightarrow{e}z\in T$ be an edge in the spanning tree oriented towards the basepoint $*$ (namely, the path between $*$ and $y$ in $T$ passes through $z$). Hence, $\pi_y=e\pi_z$. Therefore
\[
\begin{split}
    \beta.\alpha(e)&=\beta(y)^{-1}\alpha(e)\beta (z)\\
    &=\alpha(\pi_y)^{-1}\alpha(e)\alpha(\pi_z)\\
    &=\alpha(e\pi_z)^{-1}\alpha(e\pi_z)\\
    &=\Id.
\end{split}
\]
Since the action of the $0$-cochains on the $1$-cochains does not effect the local or global defects of the $1$-cochain, we can assume the $\alpha$ we were given was already $\beta.\alpha$. Namely, it is the identity for every $e$ in the tree. Then, running the homomorphism tester (Algorithm \ref{alg:hom_test}) on $\alpha|_S\colon S\to \Sym(n)$ 
is the same as running the $1$-cocycle tester (Algorithm \ref{alg:Cocycle_test}) on $\alpha$. Thus, 
\begin{equation}\label{eq:local_defects_the_same_hom_cocyc}
   \defect_{hom}(\alpha|_S)=\defect_{cocyc}(\alpha).
\end{equation}
But, the global defects may be different. If $f\colon S\to \Sym(N)$ is the closest homomorphism  to $\alpha|_S$, then we can extend $f$ to be the identity on the tree $T$ and it becomes a $1$-cocycle of $\cX$. Furthermore
\[
\begin{split}
    d_h(f,\alpha|_S)&=\Ex_{e\notin T}[d_h(f(e),\alpha(e))]\\
    &=\frac{|E(\cX)|}{|E(\cX)|-|T|}\Ex_{e\in \raE(\cX)}[d_h(f(e),\alpha(e))]\\
    &\geq d_h(f,\alpha).
\end{split}
\]
Hence, $\Defect_{hom}(\alpha|_S)\geq \Defect_{cocyc}(\alpha)$. Together with \eqref{eq:local_defects_the_same_hom_cocyc}, this proves clause $(1)$ of the Theorem.
\\

For clause $(2)$ we apply a naive operation, which would seem very wasteful at first glance. In Remark \ref{rem:thm3_is_tight}, we show that clause $(2)$ is essentially tight, and thus this process is optimal for the worst case scenarios. 

    Let $\varphi\colon \raE(\cX)\to \Sym(N)$ be the closest cocycle to $\alpha$. Even though $\alpha$ is trivial on $T$, and thus $\alpha|_S$ is in the potential homomorphisms from $\pi_1(\cX,*)$ to $\Sym(n)$, $\varphi$ is \textbf{not necessarily trivial on $T$}. To amend this, we will apply  on $\varphi$ a similar operation to the one we applied on $\alpha$ in the beginning of this proof. Namely, there is a $0$-cochain $\beta'\colon V(\cX)\to \Sym(N)$ such that $\beta'.\varphi$ is trivial on $T$. How far can $\beta'.\varphi|_S$ be from $\alpha|_S$? Recall that $\beta'.\varphi|_S(y\xrightarrow{e}z)=\varphi(\pi_y)^{-1}\varphi(e)\varphi(\pi_z)$. On the other hand, $\alpha|_S(y\xrightarrow{e}z)=\alpha(e)=\alpha(\pi_y)^{-1}\alpha(e)\alpha(\pi_z)$, since $\alpha$ is trivial in the tree. Let $\pi_y=y\xrightarrow{e_1}...\xrightarrow{e_k}*$ and $\pi_z=z\xrightarrow{e'_1}...\xrightarrow{e'_m}*$ be the paths along the tree. Then, 
    \[
    \begin{split}
d_h(\alpha|_S(e),\beta'.\varphi|_S(e))&=d_h(\alpha(\pi_y)^{-1}\alpha(e)\alpha(\pi_z),\varphi(\pi_y)^{-1}\varphi(e)\varphi(\pi_z))\\
&=d_h(\alpha(\bar e_k)...\alpha(\bar e_1)\alpha(e)\alpha(e'_1)...\alpha(e'_m),\varphi(\bar e_k)...\varphi(\bar e_1)\varphi(e)\varphi(e'_1)...\varphi(e'_m))\\
&\leq d_h(\alpha(e),\varphi (e))+\sum_{i=1}^k d_h(\alpha(e_i),\varphi(e_i))+\sum_{j=1}^m d_h(\alpha(e'_j),\varphi(e'_j))\\
&=(\diamond).
    \end{split}
    \]
    Note that in the last sum, no oriented edge in $\raE(\cX)$ appears more than once. Hence,
    \[
\begin{split}
    (\diamond)\leq \sum_{e\in \raE(\cX)}d_h(\alpha(e),\varphi(e))=|\raE(\cX)|d_h(\alpha,\varphi).
\end{split}
    \]
Therefore, 
\[
\begin{split}
    d_h(\alpha|_S,\beta'.\varphi|_S)&=\Ex_{e\notin T}[d_h(\alpha|_S(e),\beta'.\varphi|_S(e))]\\
    &\leq \Ex_{e\notin T}\left[|\raE(\cX)|d_h(\alpha,\varphi)\right]\\
    &=|\raE(\cX)|d_h(\alpha,\varphi),
\end{split}
\]
which finishes the proof.
\end{proof}

 \begin{cor}\label{cor:other_Gammas}
     Note that the proofs of Theorems \ref{thm:main2} and \ref{thm:main3} apply in a much more general setup. We could have defined homomorphism stability with $(G,d)$-coefficients: Let $\Gamma \cong\langle S|R\rangle$ be a finite presentation. Let $f\colon S\to G$ be a function, which extends to a homomorphism from $\cF_S$. By letting 
     \[
     \defect_{hom,G}(f)=\Ex_{r\in R}[d(f(r),\Id_G)]
     \]
     and 
     \[
\Defect_{hom,G}(f)=\min\left\{\Ex_{s\in S}[d(f(s),\varphi(s))]\ \middle|\ \varphi \in Hom(\Gamma,G)\right\},
     \]
     we can define that $\langle S|R\rangle$ is $\rho$-homomorphism stable with $(G,d)$-coefficients whenever 
     \[
        \Defect_{hom,G}(f)\leq \rho(\defect_{hom,G}(f)).
     \]
     The proofs of Theorems \ref{thm:main2} and \ref{thm:main3} show that this condition is equivalent to cocycle stability in the first dimension with $(G,d)$-coefficients. Since variants  of this type are extensively studied --- Hilbert--Schmidt stablity \cites{HadwinShulman,BeckerLubotzky}, Frobenius stability \cites{DGLT,LubotzkyOppenheim} and so on --- our framework may allow one to tackle them as well. 
 \end{cor}
 
In Part II of this paper, we prove the following fact:

\begin{fact}[\cite{CL_part2}]\label{fact:complete_complex_stable}
    Let $\cX$ be the \emph{complete complex} on $d$ vertices. Namely, it has a vertex set $[d]$, an edge set $\{ij\mid i\neq j\in [d]\}$ and a triangle set $\{ijk\mid i\neq j\neq k\in [d]\}$. Then $\cX$ is $\rho$-cocycle stable with $\rho(\eps)=\eps$.
\end{fact}

\begin{rem}\label{rem:thm3_is_tight}
    We now show that the bound in clause $(2)$ of Theorem \ref{thm:main3} is, up to a factor of $12$, tight. Let us consider the complete complex $\cX$ as in Fact  \ref{fact:complete_complex_stable}. Choose the spanning tree $T$ to be the path $1\to2\to3\to ...\to d$. Recall that $S=\raE(\cX)\setminus T$, and let $\alpha\colon S\to \FF_2 \cong \Sym(2)$ be the following: 
    \[
    \forall i\xrightarrow{e}j\notin T\ \colon\ \ \alpha(e)=\begin{cases}
        1 & i\leq \nicefrac{d}{2}, j>\nicefrac{d}{2},\\
         1 & j\leq \nicefrac{d}{2}, i>\nicefrac{d}{2},\\
         0 & \textrm{otherwise}.
    \end{cases}
    \]
    Note that if we extend $\alpha$  to be trivial on $T$, it is almost the characteristic function of the cut of $\cX$ separating $\{1,...,\lfloor\nicefrac{d}{2}\rfloor\}$ and the rest of $[d]$, except $\alpha(\lfloor\nicefrac{d}{2}\rfloor\to \lfloor\nicefrac{d}{2}\rfloor+1)=0$ and not $1$ (since this edge is part of $T$). It can be checked that the only triangles that are violated by $\alpha$ are those which contain the edge $\lfloor\nicefrac{d}{2}\rfloor\to \lfloor\nicefrac{d}{2}\rfloor+1$, and there are $d-2$ of them. Hence, $\defect_{cocyc}(\alpha)=\defect_{hom}(\alpha)=\frac{d-2}{\binom{d}{3}}=\frac{6}{d(d-1)}.$ 
    On the other hand, since $\pi_1(\cX,*)=1$, there  only  homomorphisms from $S$ to $\Sym(n)$ are the constant identity ones, and clearly the closest one is the one into $\Sym(2)\cong\FF_2$ which is the constant $0$ function. 
    Therefore, $$\Defect_{hom}(\alpha)=\Ex_{ij\notin T}[\alpha(ij)=1]=\frac{ \lfloor\nicefrac{d}{2}\rfloor\lceil\nicefrac{d}{2}\rceil-1}{\binom{d-1}{2}-d+1}\approx \frac{1}{2}.$$
    All in all, 
    \[
\frac{\Defect_{hom}(\alpha)}{\defect_{hom}(\alpha)}\approx \frac{\frac{1}{2}}{\frac{6}{d(d-1)}}=\frac{|\raE(\cX)|}{12},
    \]
    which, together with Fact \ref{fact:complete_complex_stable}, proves that the cocycle stability rate of $\cX$ is at least $\frac{|\raE(\cX)|}{12}$ better than the homomorphism stability of the presentation of $\pi_1(\cX,*)$ obtained by retracting $T$.

    This example is a bit weak, since there are better spanning trees that will yield a better (though, still quite a bad) deterioration parameter. We wonder whether there is a complex such that \textbf{all} of its spanning trees produce presentations with homomorphism stability rate deteriorating by $\Omega(|\raE(\cX)|)$ compared to the cocycle stability rate.
\end{rem}

\section{Alternative distributions}  \label{sec:alt_dist}
Since we presented the property testing framework in a very general fashion in Section \ref{sec:Prop_testing}, it was not natural to discuss various normalization and weighting systems other than the uniform ones. But, for other problems (cf. \cite{CL_part2}), there is a natural way of extending the definitions to include weights.

Given a polygonal complex $\cX$, we can provide with it probability distributions $\mu_i$ on $\cX(i)$. Then, the distance between the cochains $\alpha\colon \overrightarrow{\cX}(i)\to \Sym(n)$ and $\varphi\colon \overrightarrow{\cX}(i)\to \Sym(N)$ would be
\begin{equation}\label{eq:weighted_distance_cochains}
    d_h(\alpha,\varphi)=\Ex_{[x]\sim \mu_i}[d_h(\alpha(x),\varphi(x))].
\end{equation}
Note that \eqref{eq:distance_between_cochains}  is a special case of \eqref{eq:weighted_distance_cochains} by choosing $\mu_i$ to be the uniform distributions on the un-oriented cells. Furthermore, the edit distance between $\cX$-labeled graphs can be weighted according to $\mu_1$. The sampling procedure of a polygon in Algorithms \ref{alg:Covering_test} and \ref{alg:Cocycle_test} (with $i=1$) can be taken to be according to $\mu_2$. 
The correspondence of Theorem \ref{thm:main1} which was proved in Section \ref{sec:equivalence_covering_cocycles} is still valid in this generalized setup, as long as the weights on the complex $\cX$ are fixed.

Similarly, one can associate a measure $\mu_S$ on $S$ and $\mu_R$ on $R$. Then, define the distance between $f\colon S\to \Sym(n)$ and $\varphi\colon S\to \Sym(N)$ in a similar fashion to \eqref{eq:weighted_distance_cochains},
\[
d_h(f,\varphi)=\Ex_{s\sim \mu_S}[d_h(f(s),\varphi(s))].
\]
Furthermore, we can sample the relation in Algorithm \ref{alg:hom_test} according to $\mu_R$. In this generalized framework, Theorem \ref{thm:main2} can be proved as in Section \ref{sec:hom_equiv_to_cocycle} by associating with $X_{\langle S|R\rangle}$ the measures $\mu_1=\mu_S$ on its edges and $\mu_2=\mu_R$ on its polygons. There is a weighted version of Theorem \ref{thm:main3} as well, but we omit it from this discussion.

Though this generalized form does not imply any relation between the $\mu_i$'s, there is a natural choice which behaves well when perturbations are applied to the studied object. Let us focus on cocycle stability, which we now know is equivalent to the others under the appropriate corresondences of Theorems \ref{thm:main1},\ref{thm:main2} and \ref{thm:main3}. For an edge $e$ and a polygon $\pi=e_1...e_\ell$, the occurrence number of $e$ in $\pi$ is 
\[
OC(e\prec \pi)=|\{i\in \{1,...,\ell\}\mid [e_i]=[e] \}|.
\]
Note that the occurence number is independent of the chosen orientation of the polygon or edge. 
Given a distribution $\mu_2$ on un-oriented polygons of $\cX$, let us define the following weighting system on edges
\[
\begin{split}
    \forall [e]\in E(\cX)\ \colon \ \ w([e])&=\Ex_{[\pi]\sim \mu_2}\left[OC(e\prec \pi)\right].
\end{split}
\]
Namely, the weight of $[e]$ is the expected number of times it appears in a polygon sampled according to $\mu_2$.
Then, let 
\begin{equation}\label{eq:distribution_by_weights_on_polygons}
    \begin{split}
    \forall [e]\in E(\cX)\ \colon \ \ \mu_1([e])&=\frac{w([e])}{\sum_{[e']\in E(\cX)}w(e')}.
\end{split}
\end{equation}
The reason this weighting system is natural is that it behaves well under small perturbations, as the following claim shows.

\begin{claim}\label{claim:weighting_system_perturbation}
    Let $\cX$ be a polygonal complex,  let $\mu_2$ be some distribution on its polygons, and let $L$ be the maximal length of a polygon in $\cX$. Choose $\mu_1$ to be the distribution from \eqref{eq:distribution_by_weights_on_polygons}. Let $\alpha\colon \raE(\cX)\to \Sym(n)$ and $\varphi\colon \raE(\cX)\to \Sym(N)$ be two $1$-cochains. Then, $$d_h(\delta\alpha,\delta\varphi)\leq\Ex_{[\pi]\sim \mu_2}[\ell(\pi)]\cdot d_h(\alpha,\varphi) \leq L\cdot d_h(\alpha,\varphi).$$
    In particular, if two cochains are close, then their local defects are also close.
\end{claim}

\begin{proof}
    Let us calculate,
    \[
    \begin{split}
        d_h(\delta\alpha,\delta\varphi)&=\Ex_{[e_1...e_\ell]\sim \mu_2}[d_h(\delta\alpha(e_1...e_\ell),\delta\varphi(e_1...e_\ell))]\\
        &=\Ex_{[e_1...e_\ell]\sim \mu_2}[d_h(\alpha(e_1)...\alpha(e_\ell),\varphi(e_1)...\varphi(e_\ell))]\\
        &\leq \Ex_{[e_1...e_\ell]\sim \mu_2}\left[\sum_{j=1}^\ell d_h(\alpha(e_j),\varphi(e_j))\right]\\
        &=\Ex_{[\pi]\sim \mu_2}\left[\sum_{[e]\in E(\cX)} OC(e\prec \pi)d_h(\alpha(e),\varphi(e))\right]\\
        &=\sum_{[e]\in E(\cX)} \underbrace{\Ex_{[\pi]\sim \mu_2}\left[OC(e\prec \pi)\right]}_{w([e])}d_h(\alpha(e),\varphi(e))\\
        &=\left(\sum_{[e']\in E(\cX)}w(e')\right)\left(\sum_{[e]\in E(\cX)} \mu_1([e])d_h(\alpha(e),\varphi(e))\right)
    \end{split}
    \]
    Now, on the one hand, $$\sum_{[e]\in E(\cX)} \mu_1([e])d_h(\alpha(e),\varphi(e))=\Ex_{[e]\sim \mu_1}[d_h(\alpha(e),\varphi(e))]=d_h(\alpha,\varphi).$$
    On the other hand,
    \[
\sum_{[e']\in E(\cX)}w(e')=\Ex_{[\pi]\sim \mu_2}\left[\sum_{[e']\in E(\cX)}OC(e\prec \pi)\right]=\Ex_{[\pi]\sim \mu_2}[\ell(\pi)]\leq L,
\]
where $\ell(\pi)$ is the length of the polygon $\pi$, defined in Definition \ref{defn:paths}. This finishes the proof.
\end{proof}

\bibliographystyle{plain}
\bibliography{Bib}

\begin{bibdiv}
\begin{biblist}

\bib{Alon_Kriv_k-colorability}{article}{
      author={Alon, Noga},
      author={Krivelevich, Michael},
       title={Testing k-colorability},
        date={2002},
     journal={SIAM Journal on Discrete Mathematics},
      volume={15},
      number={2},
       pages={211\ndash 227},
      eprint={https://doi.org/10.1137/S0895480199358655},
         url={https://doi.org/10.1137/S0895480199358655},
}

\bib{Linial_Amit2002random}{article}{
      author={Amit, Alon},
      author={Linial, Nathan},
       title={Random graph coverings {I}: General theory and graph
  connectivity},
        date={2002},
     journal={Combinatorica},
      volume={22},
      number={1},
       pages={1\ndash 18},
}

\bib{amit2001random}{inproceedings}{
      author={Amit, Alon},
      author={Linial, Nathan},
      author={Matou{\v{s}}ek, Ji{\v{r}}{\'\i}},
      author={Rozenman, Eyal},
       title={Random lifts of graphs},
organization={Citeseer},
        date={2001},
   booktitle={Proceedings of the twelfth annual acm-siam symposium on discrete
  algorithms},
       pages={883\ndash 894},
}

\bib{PCP_thm}{article}{
      author={Arora, Sanjeev},
      author={Lund, Carsten},
      author={Motwani, Rajeev},
      author={Sudan, Madhu},
      author={Szegedy, Mario},
       title={Proof verification and the hardness of approximation problems},
        date={1998may},
        ISSN={0004-5411},
     journal={J. ACM},
      volume={45},
      number={3},
       pages={501–555},
         url={https://doi.org/10.1145/278298.278306},
}

\bib{Alon_Cheeger_2}{article}{
      author={Alon, Noga},
       title={Eigenvalues and expanders},
        date={1986},
     journal={Combinatorica},
      volume={6},
      number={2},
       pages={83\ndash 96},
         url={https://doi.org/10.1007/BF02579166},
}

\bib{ALON_Milman_Cheeger_1}{article}{
      author={Alon, N},
      author={Milman, V.D},
       title={$\lambda_1$, isoperimetric inequalities for graphs, and
  superconcentrators},
        date={1985},
        ISSN={0095-8956},
     journal={Journal of Combinatorial Theory, Series B},
      volume={38},
      number={1},
       pages={73\ndash 88},
  url={https://www.sciencedirect.com/science/article/pii/0095895685900929},
}

\bib{ArzhantsevaPaunescu}{article}{
      author={Arzhantseva, Goulnara},
      author={P\u{a}unescu, Liviu},
       title={Almost commuting permutations are near commuting permutations},
        date={2015},
        ISSN={0022-1236},
     journal={J. Funct. Anal.},
      volume={269},
      number={3},
       pages={745\ndash 757},
         url={https://doi.org/10.1016/j.jfa.2015.02.013},
}

\bib{BowenBurton}{article}{
      author={Bowen, Lewis},
      author={Burton, Peter C.~M.},
       title={Flexible stability and nonsoficity},
        date={2019},
     journal={Transactions of the American Mathematical Society},
}

\bib{BC22}{article}{
      author={Becker, Oren},
      author={Chapman, Michael},
       title={Stability of approximate group actions: uniform and
  probabilistic},
        date={2022},
     journal={J. Eur. Math. Soc.},
}

\bib{BCLV_subgroup_tests}{article}{
      author={Bowen, Lewis},
      author={Chapman, Michael},
      author={Lubotzky, Alex},
      author={Vidick, Thomas},
       title={{The Aldous--Lyons Conjecture I}: {Subgroup Tests}},
        date={2024},
     journal={preprint},
}

\bib{BFL91}{article}{
      author={Babai, L{\'a}szl{\'o}},
      author={Fortnow, Lance},
      author={Lund, Carsten},
       title={Non-deterministic exponential time has two-prover interactive
  protocols},
        date={1991},
     journal={computational complexity},
      volume={1},
      number={1},
       pages={3\ndash 40},
         url={https://doi.org/10.1007/BF01200056},
}

\bib{BeckerLubotzky}{article}{
      author={Becker, Oren},
      author={Lubotzky, Alexander},
       title={Group stability and {P}roperty ({T})},
        date={2020},
        ISSN={0022-1236},
     journal={J. Funct. Anal.},
      volume={278},
      number={1},
         url={https://doi.org/10.1016/j.jfa.2019.108298},
}

\bib{becker_Mosheiff_Luobtzky2022testability}{inproceedings}{
      author={Becker, Oren},
      author={Lubotzky, Alexander},
      author={Mosheiff, Jonathan},
       title={Testability of relations between permutations},
organization={IEEE},
        date={2022},
   booktitle={2021 {IEEE} 62nd annual symposium on foundations of computer
  science ({FOCS})},
       pages={286\ndash 297},
}

\bib{becker2022testability}{article}{
      author={Becker, Oren},
      author={Lubotzky, Alexander},
      author={Mosheiff, Jonathan},
       title={Testability in group theory},
        date={2023},
     journal={Israel Journal of Mathematics},
      volume={256},
      number={1},
       pages={61\ndash 90},
}

\bib{BLR}{inproceedings}{
      author={Blum, Manuel},
      author={Luby, Michael},
      author={Rubinfeld, Ronitt},
       title={Self-testing/correcting with applications to numerical problems},
        date={1993},
   booktitle={Proceedings of the 22nd {A}nnual {ACM} {S}ymposium on {T}heory of
  {C}omputing ({B}altimore, {MD}, 1990)},
      volume={47},
       pages={549\ndash 595},
         url={https://doi.org/10.1016/0022-0000(93)90044-W},
}

\bib{BLT}{article}{
      author={Becker, Oren},
      author={Lubotzky, Alexander},
      author={Thom, Andreas},
       title={Stability and invariant random subgroups},
        date={2019},
        ISSN={0012-7094},
     journal={Duke Math. J.},
      volume={168},
      number={12},
       pages={2207\ndash 2234},
         url={https://doi.org/10.1215/00127094-2019-0024},
}

\bib{BeckerMosheiff}{article}{
      author={Becker, Oren},
      author={Mosheiff, Jonathan},
       title={Abelian groups are polynomially stable},
        date={2021},
     journal={International Mathematics Research Notices},
      volume={2021},
      number={20},
       pages={15574\ndash 15632},
}

\bib{BOT}{article}{
      author={Burger, Marc},
      author={Ozawa, Narutaka},
      author={Thom, Andreas},
       title={On {U}lam stability},
        date={2013},
        ISSN={0021-2172},
     journal={Israel J. Math.},
      volume={193},
      number={1},
       pages={109\ndash 129},
         url={https://doi.org/10.1007/s11856-012-0050-z},
}

\bib{ben2003some}{inproceedings}{
      author={Ben-Sasson, Eli},
      author={Harsha, Prahladh},
      author={Raskhodnikova, Sofya},
       title={Some {3CNF} properties are hard to test},
        date={2003},
   booktitle={Proceedings of the thirty-fifth annual acm symposium on theory of
  computing},
       pages={345\ndash 354},
}

\bib{Cap_Lup_Sofic_Hyperlinear_book}{article}{
      author={Capraro, Valerio},
      author={Lupini, Martino},
       title={Introduction to sofic and hyperlinear groups and connes'
  embedding conjecture},
        date={2013},
         url={https://arxiv.org/abs/1309.2034},
}

\bib{CL_part2}{article}{
      author={Chapman, Michael},
      author={Lubotzky, Alex},
       title={Stability of homomorphisms, coverings and cocycles {II}:
  {Examples, Applications and Open problems}},
        date={2023},
     journal={preprint},
}

\bib{CL_stability_Random_complexes}{article}{
      author={Chapman, Michael},
      author={Peled, Yuval},
       title={Cocycle stability in permutations of random simplicial
  complexes},
        date={2023},
     journal={preprint},
}

\bib{CVY_efficient}{article}{
      author={Chapman, Michael},
      author={Vidick, Thomas},
      author={Yuen, Henry},
       title={Efficiently stable presentations from error-correcting codes},
        date={2023},
     journal={Preprint},
}

\bib{DGLT}{article}{
      author={De~Chiffre, Marcus},
      author={Glebsky, Lev},
      author={Lubotzky, Alexander},
      author={Thom, Andreas},
       title={Stability, cohomology vanishing, and non-approximable groups},
        date={2020},
     journal={Forum Math. Sigma},
      volume={8},
       pages={e18},
         url={https://doi.org/10.1017/fms.2020.5},
}

\bib{dikstein2023coboundary}{misc}{
      author={Dikstein, Yotam},
      author={Dinur, Irit},
       title={Coboundary and cosystolic expansion without dependence on
  dimension or degree},
        date={2023},
        note={\url{https://arxiv.org/abs/2304.01608}},
}

\bib{dinur2022good}{misc}{
      author={Dinur, Irit},
      author={Evra, Shai},
      author={Livne, Ron},
      author={Lubotzky, Alexander},
      author={Mozes, Shahar},
       title={Good locally testable codes},
        date={2022},
}

\bib{LTC_DELLM}{inproceedings}{
      author={Dinur, Irit},
      author={Evra, Shai},
      author={Livne, Ron},
      author={Lubotzky, Alexander},
      author={Mozes, Shahar},
       title={Locally testable codes with constant rate, distance, and
  locality},
        date={2022},
   booktitle={Proceedings of the 54th {Annual ACM SIGACT Symposium on Theory of
  Computing}},
      series={STOC 2022},
   publisher={Association for Computing Machinery},
     address={New York, NY, USA},
       pages={357–374},
         url={https://doi.org/10.1145/3519935.3520024},
}

\bib{dinur2007pcp}{article}{
      author={Dinur, Irit},
       title={The {PCP} theorem by gap amplification},
        date={2007},
     journal={Journal of the ACM},
      volume={54},
      number={3},
       pages={12\ndash es},
}

\bib{DKW_cosystolic_exp_implies_overlap}{article}{
      author={Dotterrer, Dominic},
      author={Kaufman, Tali},
      author={Wagner, Uli},
       title={On expansion and topological overlap},
        date={2018},
     journal={Geometriae Dedicata},
      volume={195},
      number={1},
       pages={307\ndash 317},
         url={https://doi.org/10.1007/s10711-017-0291-4},
}

\bib{de_la_Salle_spectral_gap}{misc}{
      author={de~la Salle, Mikael},
       title={Spectral gap and stability for groups and non-local games},
   publisher={arXiv},
        date={2022},
         url={https://arxiv.org/abs/2204.07084},
        note={\url{https://arxiv.org/abs/2204.07084}},
}

\bib{Dinur-Meshulam}{article}{
      author={Dinur, Irit},
      author={Meshulam, Roy},
       title={Near coverings and cosystolic expansion},
        date={2022},
     journal={Archiv der Mathematik},
      volume={118},
      number={5},
       pages={549\ndash 561},
         url={https://doi.org/10.1007/s00013-022-01720-6},
}

\bib{Dodziuk_Cheeger}{article}{
      author={Dodziuk, Jozef},
       title={Difference equations, isoperimetric inequality and transience of
  certain random walks},
        date={1984)},
     journal={Trans. Amer. Math. Soc.},
      number={2},
       pages={787\ndash 794},
}

\bib{dogon2023flexible}{article}{
      author={Dogon, Alon},
       title={Flexible {H}ilbert-{S}chmidt stability versus hyperlinearity for
  property ({T}) groups},
        date={2023},
        ISSN={0025-5874,1432-1823},
     journal={Math. Z.},
      volume={305},
      number={4},
       pages={Paper No. 58, 20},
         url={https://doi.org/10.1007/s00209-023-03387-3},
      review={\MR{4660281}},
}

\bib{evra2016bounded}{inproceedings}{
      author={Evra, Shai},
      author={Kaufman, Tali},
       title={Bounded degree cosystolic expanders of every dimension},
        date={2016},
   booktitle={Proceedings of the forty-eighth annual acm symposium on theory of
  computing},
       pages={36\ndash 48},
}

\bib{Gallager_LDPC}{book}{
      author={Gallager, Robert~G.},
       title={Low density parity check codes},
   publisher={MIT press},
        date={1963},
}

\bib{GowersHatami}{article}{
      author={Gowers, William~T.},
      author={Hatami, Omid},
       title={Inverse and stability theorems for approximate representations of
  finite groups},
        date={2017},
        ISSN={0368-8666},
     journal={Mat. Sb.},
      volume={208},
      number={12},
       pages={70\ndash 106},
         url={https://doi.org/10.4213/sm8872},
}

\bib{glebsky2023asymptotic}{misc}{
      author={Glebsky, Lev},
      author={Lubotzky, Alexander},
      author={Monod, Nicolas},
      author={Rangarajan, Bharatram},
       title={Asymptotic cohomology and uniform stability for lattices in
  semisimple groups},
        date={2023},
}

\bib{Goldreich}{book}{
      author={Goldreich, Oded},
       title={Introduction to property testing},
   publisher={Cambridge University Press, Cambridge},
        date={2017},
        ISBN={978-1-107-19405-2},
         url={https://doi.org/10.1017/9781108135252},
}

\bib{GlebskyRivera}{article}{
      author={Glebsky, Lev},
      author={Rivera, Luis~Manuel},
       title={Almost solutions of equations in permutations},
        date={2009},
        ISSN={1027-5487},
     journal={Taiwanese J. Math.},
      volume={13},
      number={2A},
       pages={493\ndash 500},
         url={https://doi.org/10.11650/twjm/1500405351},
}

\bib{gromov2010singularities}{article}{
      author={Gromov, Mikhail},
       title={Singularities, expanders and topology of maps. part 2: From
  combinatorics to topology via algebraic isoperimetry},
        date={2010},
     journal={Geometric and Functional Analysis},
      volume={20},
      number={2},
       pages={416\ndash 526},
}

\bib{Gromov_sofic}{article}{
      author={Gromov, Misha},
       title={Endomorphisms of symbolic algebraic varieties},
        date={1999},
     journal={J. Eur. Math. Soc.},
      volume={1},
      number={2},
       pages={109\ndash 197},
}

\bib{Hatcher_Alg_Top}{book}{
      author={Hatcher, Allen},
       title={Algebraic topology},
   publisher={Cambridge University Press},
     address={Cambridge},
        date={2002},
        ISBN={0-521-79160-X; 0-521-79540-0},
}

\bib{Hoory_Linial_Wigderson}{article}{
      author={Hoory, Shlomo},
      author={Linial, Nathan},
      author={Wigderson, Avi},
       title={Expander graphs and their applications},
        date={2006},
     journal={Bull. Amer. Math. Soc.},
      volume={43},
       pages={439\ndash 561},
}

\bib{HadwinShulman}{article}{
      author={Hadwin, Don},
      author={Shulman, Tatiana},
       title={Stability of group relations under small {H}ilbert-{S}chmidt
  perturbations},
        date={2018},
        ISSN={0022-1236},
     journal={J. Funct. Anal.},
      volume={275},
      number={4},
       pages={761\ndash 792},
         url={https://doi.org/10.1016/j.jfa.2018.05.006},
}

\bib{MIPRE}{article}{
      author={Ji, Zhengfeng},
      author={Natarajan, Anand},
      author={Vidick, Thomas},
      author={Wright, John},
      author={Yuen, Henry},
       title={{MIP*=RE}},
        date={2021},
     journal={Commun. ACM},
      volume={64},
      number={11},
       pages={131\ndash 138},
}

\bib{quantum_soundness_tensor_codes}{inproceedings}{
      author={Ji, Zhengfeng},
      author={Natarajan, Anand},
      author={Vidick, Thomas},
      author={Wright, John},
      author={Yuen, Henry},
       title={Quantum soundness of testing tensor codes},
organization={IEEE},
        date={2022},
   booktitle={2021 {IEEE} 62nd annual symposium on foundations of computer
  science ({FOCS})},
       pages={586\ndash 597},
}

\bib{Kazhdan}{article}{
      author={Kazhdan, David},
       title={On {$\varepsilon $}-representations},
        date={1982},
        ISSN={0021-2172},
     journal={Israel J. Math.},
      volume={43},
      number={4},
       pages={315\ndash 323},
         url={https://doi.org/10.1007/BF02761236},
}

\bib{kaufman2014high}{inproceedings}{
      author={Kaufman, Tali},
      author={Lubotzky, Alexander},
       title={High dimensional expanders and property testing},
        date={2014},
   booktitle={Proceedings of the 5th conference on innovations in theoretical
  computer science},
       pages={501\ndash 506},
}

\bib{KOLODNER2021595}{article}{
      author={Kolodner, Noam~M.D.},
       title={On algebraic extensions and decomposition of homomorphisms of
  free groups},
        date={2021},
        ISSN={0021-8693},
     journal={Journal of Algebra},
      volume={569},
       pages={595\ndash 615},
  url={https://www.sciencedirect.com/science/article/pii/S0021869320305408},
}

\bib{levit_lazarovich2019surface}{article}{
      author={Lazarovich, Nir},
      author={Levit, Arie},
      author={Minsky, Yair},
       title={Surface groups are flexibly stable},
        date={2019},
     journal={arXiv preprint arXiv:1901.07182},
}

\bib{linial_meshulam2006homological}{article}{
      author={Linial, Nathan},
      author={Meshulam, Roy},
       title={Homological connectivity of random 2-complexes},
        date={2006},
     journal={Combinatorica},
      volume={26},
      number={4},
       pages={475\ndash 487},
}

\bib{LubotzkyOppenheim}{article}{
      author={Lubotzky, Alexander},
      author={Oppenheim, Izhar},
       title={Non p-norm approximated groups},
        date={2020},
     journal={Journal d'Analyse Math{\'e}matique},
      volume={141},
      number={1},
       pages={305\ndash 321},
         url={https://doi.org/10.1007/s11854-020-0119-2},
}

\bib{lubotzky1994discrete}{book}{
      author={Lubotzky, Alex},
       title={Discrete groups, expanding graphs and invariant measures},
   publisher={Springer Science},
        date={1994},
      volume={125},
}

\bib{lubotzky2018high}{inproceedings}{
      author={Lubotzky, Alexander},
       title={High dimensional expanders},
   booktitle={Proceedings of the {International Congress of Mathematicians: Rio
  de Janeiro 2018}},
       pages={705\ndash 730},
        note={World Scientific, 2018},
}

\bib{MSS_Ramanujan}{inproceedings}{
      author={Marcus, Adam},
      author={Spielman, Daniel~A.},
      author={Srivastava, Nikhil},
       title={{Interlacing Families I}: Bipartite ramanujan graphs of all
  degrees},
        date={2013},
   booktitle={2013 {IEEE} 54th annual symposium on foundations of computer
  science},
       pages={529\ndash 537},
}

\bib{odonnell2021analysis}{book}{
      author={O'Donnell, Ryan},
       title={Analysis of boolean functions},
        date={2021},
        note={\url{https://arxiv.org/abs/2105.10386}},
}

\bib{LTC_Panteleev_Kalachev}{inproceedings}{
      author={Panteleev, Pavel},
      author={Kalachev, Gleb},
       title={Asymptotically good quantum and locally testable classical ldpc
  codes},
        date={2022},
   booktitle={Proceedings of the 54th annual acm sigact symposium on theory of
  computing},
      series={STOC 2022},
   publisher={Association for Computing Machinery},
     address={New York, NY, USA},
       pages={375–388},
         url={https://doi.org/10.1145/3519935.3520017},
}

\bib{Puder_expander_coverings}{article}{
      author={Puder, Doron},
       title={Expansion of random graphs: new proofs, new results},
        date={2015},
     journal={Inventiones mathematicae},
      volume={201},
      number={3},
       pages={845\ndash 908},
         url={https://doi.org/10.1007/s00222-014-0560-x},
}

\bib{Raz_parallel_rep_survey}{inproceedings}{
      author={Raz, Ran},
       title={Parallel repetition of two prover games (invited survey)},
        date={2010},
   booktitle={2010 {IEEE} 25th annual conference on computational complexity},
       pages={3\ndash 6},
}

\bib{Raz_Parallel_Rep_paper}{article}{
      author={Raz, Ran},
       title={A parallel repetition theorem},
        date={1998},
     journal={SIAM Journal on Computing},
      volume={27},
      number={3},
       pages={763\ndash 803},
}

\bib{radhakrishnan2007dinur}{article}{
      author={Radhakrishnan, Jaikumar},
      author={Sudan, Madhu},
       title={On {Dinur's} proof of the {PCP} theorem},
        date={2007},
     journal={Bulletin of the American Mathematical Society},
      volume={44},
      number={1},
       pages={19\ndash 61},
}

\bib{sageev1995ends}{article}{
      author={Sageev, Michah},
       title={Ends of group pairs and non-positively curved cube complexes},
        date={1995},
     journal={Proceedings of the London Mathematical Society},
      volume={3},
      number={3},
       pages={585\ndash 617},
}

\bib{Trees_Serre}{book}{
      author={Serre, Jean-Pierre},
       title={Trees},
   publisher={Springer},
        date={1980},
}

\bib{slofstra_2019}{article}{
      author={Slofstra, William},
       title={The set of quantum correlations is not closed},
        date={2019},
     journal={Forum of Mathematics, Pi},
      volume={7},
       pages={e1},
}

\bib{Sipser_Spielman_exp_codes}{article}{
      author={Sipser, M.},
      author={Spielman, D.A.},
       title={Expander codes},
        date={1996},
     journal={{IEEE} Transactions on Information Theory},
      volume={42},
      number={6},
       pages={1710\ndash 1722},
}

\bib{Ulam}{book}{
      author={Ulam, Stanis{\l}aw~M.},
       title={A collection of mathematical problems},
      series={Interscience Tracts in Pure and Applied Mathematics, no. 8},
   publisher={Interscience Publishers, New York-London},
        date={1960},
}

\bib{Weiss_Sofic}{article}{
      author={Weiss, Benjamin},
       title={Sofic groups and dynamical systems},
        date={2000},
        ISSN={0581572X},
     journal={Sankhya: The Indian Journal of Statistics, Series A (1961-2002)},
      volume={62},
      number={3},
       pages={350\ndash 359},
         url={http://www.jstor.org/stable/25051326},
}

\end{biblist}
\end{bibdiv}

\end{document}